\documentclass[11pt]{amsart}
\usepackage{amsmath,amssymb,amsfonts,amsthm, amscd,indentfirst}
\usepackage{amsmath,latexsym,amssymb,amsmath,%awheremscd,
	amscd,amsthm,amsxtra}
\usepackage{hyperref}\usepackage{url}
\usepackage{color}
\usepackage{pdfpages}
\usepackage{graphics}
\usepackage{enumitem}

\usepackage{graphicx}
\usepackage{caption}

\usepackage{epsfig,here}
\usepackage{subfigure,here}

\newtheorem{definition}{Definition}[section]
\newtheorem{theorem}{Theorem}[section]
\newtheorem{prop}{Proposition}[section]
\newtheorem{lemma}{Lemma}[section]

\newtheorem{corollary}{Corollary}[section]
\def\rr{\mathbb{R}}
\def\ss{\mathbb{S}}

\def\tr{\mathrm{tr}}
\def\o{\omega}

\def\p{\partial}

\def\a{\alpha}

\def\p{\partial}

\def\S{{\Sigma}}
\def\<{\langle}
\def\>{\rangle}
\def\div{{\rm div}}
\def\n{\nabla}

\def\ep{\epsilon}

\numberwithin{equation} {section}

\begin{document}
	
\title[Stable anisotropic capillary hypersurfaces in a Half-Space]{Stable anisotropic capillary hypersurfaces\\ in a Half-Space}
\author{Jinyu Guo}
\address{Department of Mathematical Sciences, Tsinghua University, Beijing, 100084, China}
\email{guojinyu@tsinghua.edu.cn}
\author{Chao Xia}
\address{School of Mathematical Sciences\\
Xiamen University\\
361005, Xiamen, P.R. China}
\email{chaoxia@xmu.edu.cn}
\thanks{JG is supported by Shuimu Tsinghua Scholar Program (No.2022SM046) and China Postdoctoral Science Foundation (No.2022M720079). CX is supported by the NSFC (Grant No.11871406, 12271449).}
	
	\begin{abstract}
		In this paper, we study stability problem of anisotropic capillary hypersurfaces in an Euclidean half-space. We prove that any compact immersed anisotropic capillary constant anisotropic mean curvature hypersurface in the half-space is weakly stable if and only if it is a truncated Wulff shape. On the other hand, we prove a Bernstein-type theorem for stable anisotropic capillary minimal surfaces in the three dimensional half-space under Euclidean area growth assumption.
\end{abstract}
	
	\date{}
	
	\maketitle
	
	\medskip
	
	%\tableofcontents
	
	%{\bf Keywords:} Capillary surfaces, Stability, Rigidity, Surface with free boundary.
\section{Introduction}

%When $F\equiv1$ and $\partial\Sigma=\emptyset$, the study of CAMC immersions becomes the classical case of CMC immersions.
%A classical stability result of constant mean curvature (CMC) hypersurfaces was first studied by Barbosa-do Carmo \cite{BdC}. They proved that: {\it any stable immersed closed CMC hypersurfaces in $\mathbb{R}^{n+1}$ are spheres.} A closed CMC hypersurface is called stable if the second variation of the area functional is nonnegative among any volume-preserving variations. Subsequently, their results were generalized to general anisotropic version by Palmer \cite{Pa} (also see Winklmann \cite{Wink}) says that: {\it any stable immersed closed constant anisotropic mean curvature (CAMC) hypersurface in $\mathbb{R}^{n+1}$ is Wulff shape (up to translations and homotheties).}  A closed CAMC hypersurface is called stable if the second variation of the anisotropic area functional (see Section \ref{sec3}) is nonnegative among any volume-preserving variations. In \cite{HL1}, He-Li considered the stability problem of $r$-th CAMC case and proved some rigidity results. Note that, plenty of other results for CAMC surfaces have been studied, see e.g. \cite{GM, HL0, HLMG, KP00, KP4, KP} and references therein. All those results are for closed hypersurfaces.

Capillary phenomena appear in the study of the equilibrium shape of liquid drops and crystals
in a given solid container. The mathematical model has been established through the work
of Young, Laplace, Gauss and others, as a variational problem on minimizing a free energy
functional under a volume constraint. For more detailed description on
the isotropic and anisotropic capillary phenomena, we refer to \cite{Finn} and \cite{DPM}.

A capillary hypersurface in a domain $B$ of a Riemannian manifold is an immersed  constant mean curvature (CMC) hypersurface in $B$ which intersects $\partial B$ at a constant contact angle. Capillary hypersurfaces are the stationary points of an energy functional under volume preserving variation. As a special and important case,  free boundary CMC hypersurfaces are the stationary points of an area functional under volume preserving variation. A capillary hypersurface is called weakly stable if it is local energy-minimizing under volume preserving variation. The stability of free boundary CMC or capillary hypersurfaces was initiated by Ros-Vergasta \cite{RV} and Ros-Souam \cite{RS}. The classification problem for stable compact free boundary CMC or capillary hypersurfaces in an Euclidean ball and in an Euclidean half-space $\rr^{n+1}_+$ have been studied intensively, see for instance \cite{ M1,Nu,Ba, AS, Souam2, LiXiong2}. The classification has been completed eventually by Wang and the second-named author \cite{WX} for the Euclidean ball case and Souam \cite{Souam3} for the Euclidean half-space case respectively. This generalizes the classical result of Barbosa-do Carmo \cite{BdC} on the classfication of stable closed CMC hypersurfaces in $\rr^{n+1}$.  See also recent work \cite{GWX3, XZ}.

A modern formulation of Gauss' model of capillary phenomena includes a possibly anisotropic surface tension density, which we are interested in this paper.

Let $F: \mathbb{S}^n \rightarrow \mathbb{R}_{+}$ be a positive smooth function such that the matrix $(D^{2}F+F {\rm Id})$ is positive definite,
where $D^{2}F$ is the Hessian of $F$ and ${\rm Id}$ denotes the identity matrix. $F$ determines a unique strictly convex hypersurface $\mathcal{W}_F$ in $\rr^{n+1}$, which is called the Wulff shape.
For a closed hypersurface $\S$ immersed in $\rr^{n+1}$, the anisotropic area functional is given by
$$\mathcal{A}_F(\S)=\int_\S F(\nu) dA.$$ The well-known Wulff theorem (see for example \cite{FM, Ta2}) says that the  Wulff shape (up to translation and homothety) is the global minimizer to the anisotropic area functional under fixed volume constraint. From variational point of view, the stationary points for the anisotropic area functional under volume-preserving variations are closed hypersurfaces with constant anisotropic mean curvature (CAMC). Palmer \cite{Pa} (see also Winklmann \cite{Wink}) proved that Wulff shape  (up to translation and homothety) is the only stable CAMC hypersurfaces, which is the anisotropic counterpart of Barbosa-do Carmo's \cite{BdC} result. For more rigidity results on closed CAMC hypersurfaces and related problems, we refer to \cite{GM, HL1, HLMG, KP00, KP4}.

For a compact, orientable hypersurface $\S$ immersed in some container $B\subset\mathbb{R}^{n+1}$ with boundary $\p\S$, which intersects $\p B$ transversely, the anisotropic energy functional is given by
$$\mathcal{E}_F(\S)=\mathcal{A}_F(\S)+\o_0\mathcal{A}_W(\S),$$
where $\o_0\in \rr$ is a real number,  $\mathcal{A}_W(\S)$ is so-called wetting area.
We remark that throughout this paper, $\S$ will be always orientable.

The global minimizer of $\mathcal{E}_F$ under fixed volume constraint has been characterized by Winterbottom \cite{Wint} (see also \cite{KP2}) as a truncated Wulff
shape (it is also called a Winterbottom shape), which can be viewed as the capillary counterpart of Wulff shape.
For anisotropic free energy functionals involving a gravitational potential energy term, the existence,
the regularity and boundary regularity for global or local minimizers have been studied by De
Giorgi \cite{DG}, Almgren \cite{Alm}, Taylor \cite{Ta} and De Philippis-Maggi \cite{DPM, DPM2}. For the symmetry and uniqueness of global minimizers, we refer to the work of Baer \cite{Bae} for a class of $F$ with certain symmetry and the work of Gonzalez \cite{Gon} in the isotropic case via
a symmetrization technique.

From variational point of view, the stationary points of $\mathcal{E}_F$ under fixed volume constraint are the anisotropic capillary CAMC hypersurfaces, which are CAMC and satisfy an anisotropic capillary condition (see \eqref{angle} below). In a series of papers \cite{ KP2, KP1, KP}, Koiso-Palmer studied the anisotropic capillary hypersurfaces in a slab (the domain bounded by two parallel hyperplane) and their stabilities.
In Koiso-Palmer's paper, the second variation of $\mathcal{E}_F$ is derived for a class of anisotropy satisfying certain symmetric condition in two dimensions.
In this paper, we first compute  the second variation of $\mathcal{E}_F$ for any anisotropic $F$ in any dimensions. Hence we give a clear characterization for stability.
\begin{prop} An anisotropic capillary CAMC immersion $x: \Sigma\rightarrow B\subset\rr^{n+1}$ is weakly stable if and only if
\begin{eqnarray}\label{weakly-stab-ineq}
\quad-\int_\Sigma(\div_{\Sigma}(A_{F}\nabla f)+\langle A_{F}\circ d\nu,d\nu\rangle f)f\,dA+\int_{\p \Sigma} \left(\langle A_{F}\nabla f,\mu\rangle-q_{F}f\right)f\,ds\ge 0,
\end{eqnarray}
for any $f\in C^\infty(\S)$ satisfying $\int_{\Sigma} f\, dA=0$. Here $q_F$ is given in \eqref{qF} below.
\end{prop}
For some notations involved in the above stability inequality \eqref{weakly-stab-ineq}, we refer to Section \ref{sec3}.

Recently, Koiso \cite{Koiso19, Koiso1} studied the stability problem of anisotropic capillary CAMC hypersurfaces in a wedge. In particular, she proved that a compact stable immersed anisotropic capillary CAMC hypersurface $\S$ in $\mathbb{R}^{n+1}_{+}$ with boundary $\p\S$ must be a truncated Wulff shape, provided $\p\S\subset \rr^n$ is embedded for $n=2$ and is convex for $n\geq3$. This extends the result of Choe-Koiso \cite{CK} in the isotropic case.
As we mentioned, in the isotropic case, Souam \cite{Souam3} classified stable capillary hypersurfaces in $\mathbb{R}^{n+1}_{+}$ without any additional assumption.
It is natural to ask whether the embeddedness condition for $n=2$ and the convexity condition for $n\geq3$ in Koiso's result \cite{Koiso19, Koiso1} can be removed for the classification.  In this paper we will give an affirmative answer and our main result is the following.
	\begin{theorem}[{\bf Theorem \ref{thmm4.1}}]\label{thm0.1}
	A compact, immersed  anisotropic capillary CAMC hypersurface in $\mathbb{R}^{n+1}_{+}$ is weakly stable if and only if it is a truncated Wulff shape, up to translation and homothety.
	\end{theorem}
As a special case, we have a classification for stable anisotropic free boundary CAMC hypersurfaces.
	\begin{corollary}\label{cor0.1}
	A compact, immersed anisotropic free boundary CAMC hypersurface in $\mathbb{R}^{n+1}_{+}$ is weakly stable if and only if it is a truncated Wulff shape, up to translation and homothety.
	\end{corollary}

%In case $F\equiv1$, Theorem \ref{thm0.1} reduces to Theorem \ref{pre0.1}.
The proof of Theorem \ref{thm0.1} is based on the stability inequality \eqref{weakly-stab-ineq} and the following Minkowski-type formula
\begin{equation}\label{aniso-Mink}
\int_{\Sigma} \left[n(F(\nu)+\omega_{0}\langle E^{F}_{n+1},\nu\rangle)-H_{F}\langle x,\nu\rangle\right] dA=0,
\end{equation}
where $E_{n+1}^F \in \mathbb{R}^{n+1}$ is a constant vector given in \eqref{En+1F} below.
%where $\omega_{0}\in(-F(E_{n+1}),F(-E_{n+1}))$ and $E_{n+1}^F \in \mathbb{R}^{n+1}$ is a constant vector defined by
%\begin{eqnarray}\label{En+1F}
%E_{n+1}^F= \begin{cases}\quad\frac{\Phi\left(E_{n+1}\right)}{F\left(E_{n+1}\right)} & \text { if } \omega_0<0,\\
%-\frac{\Phi\left(-E_{n+1}\right)}{F\left(-E_{n+1}\right)} & \text { if } \omega_0 \geq 0.
%\end{cases}
%\end{eqnarray}
 Formula \eqref{aniso-Mink} has been proved by Jia, Wang, Zhang and the second-named author  \cite[Theorem 1.3]{JWXZ}, which was used to prove an Alexandrov-type theorem for embedded anisotropic capillary hypersurfaces. It is an standard approach to apply Minkowski-type formula involving no boundary terms to handle the stability for free boundary or capillary problems, see for example \cite{ AS, GWX3, Souam3, WX}.

In the second part of this paper, we are interested in (strongly) stable anisotropic capillary minimal surfaces in $\mathbb{R}^{3}_{+}$. The second variational formula gives the following characterization of strong stability for anisotropic capillary minimal hypersurfaces in $B\subset\rr^{n+1}$.
\begin{prop} An anisotropic capillary minimal immersion $x: \Sigma\rightarrow B\subset\rr^{n+1}$ is (strongly) stable if and only if
\begin{eqnarray}\label{stab-ineq-x}
\quad-\int_\Sigma(\div_{\Sigma}(A_{F}\nabla f)+\langle A_{F}\circ d\nu,d\nu\rangle f)f\,dA+\int_{\p \Sigma} \left(\langle A_{F}\nabla f,\mu\rangle-q_{F}f\right)f\,ds\ge 0,
\end{eqnarray}
for any $f\in C_c^\infty(\S)$.
\end{prop}

A classical Bernstein theorem, proved  Fischer-Colbrie-Schoen \cite{FS}, do Carmo-Peng \cite{DP} and Pogorelov \cite{Po} independently, says that the only complete stable minimal surfaces in $\mathbb{R}^{3}$ are flat. Quite recently, Chodosh-Li \cite{CL} resolved a well-known conjecture of Schoen \cite[Conjecture 2.12]{CM} that  the only complete stable minimal hypersurfaces in $\mathbb{R}^{4}$ are flat, see also Catino-Mastrolia-Roncoroni \cite{CMR} for another proof. Schoen-Simon-Yau \cite{SSY} have shown that  any complete  stable minimal hypersurfaces in $\mathbb{R}^{n+1}$ with $n+1\leq6$ with Euclidean area growth
must be flat.  Bernstein-type theorem for the anisotropic case in $\rr^3$ has been proved by White \cite{White} under the Euclidean area growth assumption, by Lin \cite{Lin} when  $F$ is $C^{2}$-close to $1$. Under similar assumptions, Bernstein-type theorem for the anisotropic case has been studied by Simon \cite{Simon}, Winklmann \cite{Wink1}, Chodosh-Li \cite{CL}. It is still open question whether these extra assumptions could be removed, even in $\rr^3$. We refer to Chodosh-Li's paper \cite{CL2} on recent progress for stable anisotropic minimal hypersurfaces.

Initiated by the min-max construction for capillary minimal surfaces, Bernstein-type theorem for capillary minimal surfaces in $\mathbb{R}^{3}_{+}$ has recently attracted much attentions.
In particular, Li-Zhou-Zhu \cite{LZZ} and De Masi-De Philippis \cite{DD}, independently, proved that any properly immersed stable capillary minimal surfaces in $\mathbb{R}^{3}_{+}$ with quadratic area growth must be a half-plane. By using Fischer-Colbrie-Schoen's technique, Hong-Saturnino %\footnote{Their original statement: Let $\Sigma$ be a noncompact capillary surface immersed in $\mathbb{R}^{3}_{+}$ at constant angle $\theta$. If $H_{\Sigma} \cos \theta \geq 0$. Then $\Sigma$ is weakly stable if and only if it is a half-plane.}
\cite{HongS} proved the Bernstein-type theorem without Euclidean area growth assumption.

It is natural to consider the case of anisotropic capillary minimal surfaces. Here we prove the following Bernstein-type theorem for stable anisotropic capillary minimal surfaces in $\mathbb{R}_{+}^3$.
\begin{theorem}[{\bf Theorem \ref{thm-Bernstein-3}}]\label{thm-Bernstein} Let $\Sigma$ be an immersed anisotropic capillary minimal surface in $\mathbb{R}_{+}^3$. Assume that $\Sigma$ has  Euclidean area growth, that is, there exists some $C>0$ such that
\begin{equation}\label{vol-growth}
\operatorname{Area}\left(\Sigma \cap B_r(0)\right)<C r^2
\end{equation}
for any $r>0$. Then $\Sigma$ is stable if and only if $\Sigma$ is a half-plane.
\end{theorem}
In case $F\equiv1$, Theorem \ref{thm-Bernstein} reduces to \cite[Theorem 0.2]{LZZ} or \cite[Theorem 6.3]{DD}.

\

The remaining part of this paper is organized as follows. In Section \ref{sec2} we review some definitions and notations about anisotropic geometry. In Section \ref{sec3}
we calculate the first and the second variation formula of anisotropic energy functional in a general domain. In Section \ref{sec4} we present some useful geometric formulas for anisotropic capillary hypersurfaces in $\mathbb{R}^{n+1}_{+}$ and prove Theorem \ref{thm0.1}. In Section \ref{sec6} we discuss the stability of noncompact anisotropic capillary minimal surface in $\mathbb{R}^{3}_{+}$ and prove Theorem \ref{thm-Bernstein}.
%by using a Minkowski type identity \eqref{aniso-Mink} as admissible test function.%A proof of the first and the second variation formula  is provided in the Appendix \ref{app}.

\

\section{Notations and Preliminaries}\label{sec2}

Let $F: \mathbb{S}^n \rightarrow \mathbb{R}_{+}$ be a positive smooth function. Denote by $DF$ and $D^{2}F$ the gradient and Hessian of $F$ on $\mathbb{S}^{n}$. Then we require the matrix
 \begin{equation}\label{DFsur}
  A_{F}:=(D^{2}F+F {\rm Id})|_{x}>0\,\,\quad\text{for any }\,\,x\in\mathbb{S}^{n},
 \end{equation}
where ${\rm Id}$ denotes the identity on $T_{x}\mathbb{S}^{n}$ and $``> "$ means the matrix is positive definite. %try to consider the $\varphi$

We define the map
\begin{eqnarray}\label{Phi-1}
\Phi: \mathbb{S}^{n}&\longrightarrow& \mathbb{R}^{n+1}
\\ x&\mapsto& F(x)x+DF(x)\nonumber
\end{eqnarray}
whose image $\mathcal{W}_F=\Phi(\mathbb{S}^{n})$ is a smooth strictly convex hypersurface in $\mathbb{R}^{n+1}$ called the Wulff shape.

We may regard $F$ as a convex function on $\rr^{n+1}$ by one-homogenous extension of $F$. Precisely, set $F(x)=|x|F(\frac{x}{|x|})$ when $x\neq 0$ and $F(0)=0$, the new $F:\rr^{n+1}\to \rr$ is a one-homogenous function on $\rr^{n+1}$ which is $C^2\in (\rr^{n+1}\setminus \{0\})$. We use $\bar \nabla$ to denote the Euclidean corvariant derivative of $F$.
It is standard to see that for $x\in \ss^n$ and $V,W\in T_x\ss^n$, we have \begin{eqnarray}
&&\bar \nabla F(x)=DF(x)+F(x)x= \Phi(x), \label{euc-sph-der1}\\
&&\bar \nabla^2 F(x)(V,W)= (D^2 F+F{\rm Id})(x)(V,W)=A_F(x)(V, W).\label{euc-sph-der2}
\end{eqnarray}

Let $B$ be a closed region in an $(n+1)$-dimensional Euclidean space $\mathbb{R}^{n+1}$ with smooth boundary $\p B$. Let $x:\Sigma \rightarrow B$ be an isometric immersion from a $n$-dimensional smooth manifold $\S$ such that $\partial \Sigma \subset \partial B$. We denote by $\bar \n$, $\bar \Delta$ and $\bar \n^2$ the gradient, the Laplacian and the Hessian on $\rr^{n+1}$ respectively, while by $\n$, $\Delta$ and $\n^2$ the gradient, the Laplacian and the Hessian on $\Sigma$ respectively.  Let $T(\p \S)$ and $N(\p\S)$ be the tangent bundle and the normal bundle of $\p \S$ as a co-dimensional two submainfolds in $\rr^{n+1}$. We will use the following terminology for four normal vector fields.
We choose one of the unit normal vector field along $x$ and denote it by $\nu$.
We denote  by $\bar N$ the unit outward normal to $\p B$ in $B$ and $\mu$ be the unit outward normal to $\p \Sigma$ in $\Sigma$.
Let $\bar \nu$ be the unit normal to $\p \Sigma$ in $\p B$ such that the bases $\{\nu, \mu\}$ and $\{\bar \nu, \bar N\}$ have the same orientation in $N(\p \Sigma)$.
Hence, in $N(\p\Sigma)$,  the following relations hold:
\begin{eqnarray}
&&\mu=-\langle\nu,\bar{N}\rangle\bar{\nu}+\langle\mu,\bar{N}\rangle\bar{N},\label{mu-0}
\\&&\nu=\langle\mu,\bar{N}\rangle\bar{\nu}+\langle\nu,\bar{N}\rangle\bar{N}.\label{nu-0}
\end{eqnarray}
Equivalently,
 \begin{eqnarray}
&&\bar{\nu}=-\langle\nu,\bar{N}\rangle\mu+\langle\mu,\bar{N}\rangle\nu,\label{nubar-0}\\
&&\bar{N}=\langle\mu,\bar{N}\rangle\mu+\langle\nu,\bar{N}\rangle\nu.\label{Nbar-0}
\end{eqnarray}
We always assume that $\S$ interesects $\p B$ transversally, so that $\langle\mu,\bar{N}\rangle\neq 0$.
For a vector field $Y$ on $\rr^{n+1}$, we denote $Y^{\S}$ and $Y^{\p\S}$ to be the tangential projection of $Y$ on $T\S$ and on $T(\p\S)$ respectively.

%See Figure 1. %\ref{fig1}.
%We use $\langle\cdot, \cdot\rangle$ to denote the Euclidean inner product

Denote by $h$ and $H$ the second fundamental form and the mean curvature of the immersion $x$ respectively. Precisely,
$h(X, Y)= \langle\bar \n_X \nu, Y\rangle$ for $X, Y\in T\S$ and $H=\tr_g(h).$ Denote by $h^{\p B}$ the second fundamental form of $\p B$ in $B$, that is, $h^{\p B}(X, Y)= \langle\bar \n_X \bar{N}, Y\rangle$ for $X, Y\in T(\p B)$.

%We make a choice of the normal $\nu$ so that, along $x(\partial \Sigma)$, the angle between $-\nu$ and $\bar{N}$ is equivalent to the angle between $\mu$ and $\bar{\nu}$ in everywhere.

Let $\nu_{F}$ be the anisotropic normal of $\Sigma$ given by
\begin{equation}\label{nuf}
\nu_F=\Phi(\nu):=DF(\nu)+F(\nu) \nu.
\end{equation}
Hence $DF(\nu)=\nu_{F}-\langle\nu_{F},\nu\rangle\nu =\nu_F^\S\in T\Sigma$. The anisotropic principal curvatures $\left\{\kappa_i^F\right\}_{i=1}^n$ of $\Sigma$ are given by the eigenvalues of the anisotropic Weingarten map
\begin{equation}\label{weingerten}
 d\nu_F=A_F(\nu) \circ d \nu: T \Sigma \rightarrow T \Sigma.
\end{equation}
The eigenvalues are real since $\left(A_F\right)$ is positive definite and symmetric. The anisotropic second fundamental form and anisotropic mean curvature are denoted respectively by
\begin{eqnarray}
&&h_{F}(X,Y)=\langle\bar{\nabla}_{X}\nu_F,Y\rangle=(A_{F}\circ h)(X,Y),\label{second-h}\\
&&H_{F}=\tr_{g}(d\nu_{F})=\tr_{g}(A_{F}(\nu)\circ d\nu)=\div_{\Sigma}(DF(\nu))+HF(\nu).\label{HF}
\end{eqnarray}
When $F\equiv1$, we see $A_F=Id_{\ss^n}$ and hence $h_{F}$ and $H_{F}$ are the usual second fundamental form $h$ and mean curvature $H$ respectively.

\

\section{The first and second variational formula}\label{sec3}
Let $x:\Sigma \rightarrow B$ be an isometric immersion such that $x(\partial \Sigma)=x(\Sigma)\cap\partial B$. By a compactly supported admissible variation of $x$, we mean a differentiable map $x: (-\epsilon, \epsilon)\times \Sigma\to B$ such that $x(t, \cdot): \Sigma\to B$ is an immersion satisfying $x(t, {\rm int} \Sigma)\subset{\rm int} B$, $x(t, \p \Sigma)\subset\p B$ and the support of $\frac{\p}{\p t}x(t, \cdot)$ is compact for every $t\in (-\ep, \ep)$ and $x(0, \cdot)=x$.

The anisotropic area functional $\mathcal{A}_{F}: (-\ep, \ep)\to\rr$ and the oriented volume functional $\mathcal{V}: (-\ep, \ep)\to\rr$ are given by
\begin{eqnarray*}
&&\mathcal{A}_F(t)=\int_{\tilde{\Sigma}} F(\nu)dA_{t},\\
&&\mathcal{V}(t)=\int_{[0,t]\times \tilde{\Sigma}} x(t,\cdot)^*dV,
\end{eqnarray*}
where $\tilde{\Sigma}\subseteq\Sigma$ is the support of the variation $\frac{\p}{\p t}x(t, \cdot)$, $dA_{t}$ is the area element of $\Sigma$ with respect to the metric induced by $x(t, \cdot)$ and $dV$ is the volume element in $B$. When $\S$ is compact, then $\tilde{\Sigma}=\Sigma$.
An admissible variation is said to be volume-preserving if $\mathcal{V}(t)=\mathcal{V}(0)=0$ for each $t\in (-\ep, \ep)$.

The wetting area functional $\mathcal{A}_{W}(t): (-\epsilon, \epsilon)\rightarrow \mathbb{R}$ are defined by
\begin{equation}\label{wetting-1}
    \mathcal{A}_{W}(t)=\int_{[0,t]\times(\partial\Sigma\cap\tilde{\Sigma})}x(t,\cdot)^{\ast}dA_{\partial B},
\end{equation}
where $dA_{\partial B}$ is the area element of $\partial B$.

%Denote $E_{n+1}=(0,\cdots,0,1)\in\mathbb{R}^{n+1}$.
Fix a real number $\omega_{0}\in \rr$. The anisotropic energy functional $\mathcal{E}_{F}(t): (-\ep, \ep)\to\rr$ is defined by
\begin{equation}\label{energy-1}
\mathcal{E}_F(t)=\mathcal{A}_F(t)+\omega_{0} \mathcal{A}_{W}(t).
\end{equation}

It is well known that the first variation formulas of $\mathcal{V}$ and $\mathcal{A}_{W}$ for a compactly supported admissible variation with variation vector field $Y=\frac{\p}{\p t}x(t,\cdot)|_{t=0}\in C_c^\infty(\S)$ such that $Y|_{\p\S}\in T(\p B)$ are given by
\begin{eqnarray}
&&\mathcal{V}'(0)=\int_\Sigma \langle Y, \nu\rangle dA,\label{var-V}\\
&&\mathcal{A}_W'(0)=\int_{\p\Sigma} \langle Y, \bar{\nu}\rangle ds,\label{var-AW}
%&&\mathcal{E}'(0)=\int_\Sigma H^{F}\langle Y, \nu\rangle dA+\int_{\p \Sigma} \langle Y, \mu_{F}+\omega_{0}\bar{\nu}\rangle ds,
\end{eqnarray} where $dA$ and $ds$ are the area element of $\Sigma$ and $\p \Sigma$ respectively.
Next we derive the first variational formula for $\mathcal{E}_F$.
\begin{prop}\label{first-var}Let $x(\cdot, t)$ be a compactly supported admissible variation with variational vector field $Y$. Then %having $f\nu$ as its normal part. Then
\begin{equation}\label{first-varformula}
 \mathcal{E}_F'(0)=\int_\Sigma H_{F}\langle Y, \nu\rangle dA+\int_{\p \Sigma} \langle Y, \mu_{F}+\omega_{0}\bar{\nu}\rangle ds,
\end{equation}
where
\begin{eqnarray}
&&\mu_{F}:=\<\nu_F,\nu\>\mu-\langle\nu_{F},\mu\rangle\nu=F(\nu)\mu-\langle\nu_{F},\mu\rangle\nu\in N(\p\S). \label{muF}
\end{eqnarray}
\end{prop}
The first variational formula \eqref{first-varformula} is known to experts. For completeness, we will give a proof of Proposition \ref{first-var} in Appendix \ref{app}.

An interesting property for $\mu_F$ is as follows.
\begin{prop}Along $\p\S$, we have\begin{eqnarray}
\<\mu_F, \bar \nu\>&=&- \langle\nu_{F},\bar{N}\rangle,\label{xeq2}\\
\<\mu_F, \bar{N}\>&=&\langle\nu_{F},\bar \nu\rangle.\label{xeq3}
\end{eqnarray}
\end{prop}
\begin{proof}
It follows directly by the definition of $\mu_F$, $\nu_F$ and also the relations \eqref{nubar-0} and \eqref{Nbar-0}.
\end{proof}
\begin{prop}\label{Def-CAMC}
Stationary points of $\mathcal{E}_F$ under compactly supported volume preserving admissible variation are hypersurfaces with constant anisotropic mean curvature (CAMC)  satisfying the {\it anisotropic capillary} boundary condition
\begin{equation}\label{angle}
  \langle\nu_{F},\bar{N}\rangle=\omega_{0}\quad\text{along}\,\,\partial\Sigma.
\end{equation}
Similarly, stationary points of $\mathcal{E}_F$ under any compactly supported admissible variation are hypersurfaces with zero anisotropic mean curvature and satisfying \eqref{angle}.\end{prop}
\begin{proof}
From the first variation formulas  \eqref{first-varformula} and \eqref{var-V},  it is clear $x$ has constant anisotropic mean curvature. Next we show \eqref{angle}. From the first variation formulas  \eqref{first-varformula} and admissible condition, we have
\begin{eqnarray}\label{capill}
\langle Y, \mu_{F}+\omega_{0}\bar{\nu}\rangle=0 \hbox{ for any }Y\in T({\p B}).
\end{eqnarray}
Note that $\mu_F+\o_0\bar \nu\in N(\p \S)$. Since any $Y\in T({\p B})$ can be split as $Y=\<Y,\bar \nu\>\bar \nu+Y^{\p\S}$, where $Y^{\p\S}\in T(\p\S)$,  we see that
\eqref{capill} is equivalent that \begin{eqnarray}\label{capill2}
\<\bar \nu, \mu_F\>=-\o_0 \quad\text{along}\,\,\partial\Sigma.
\end{eqnarray}
From \eqref{xeq2}, we see
\eqref{capill2} is equivalent to \eqref{angle}.
\end{proof}

%H_{F}=0\quad\text{and}\quad
%\begin{equation}\label{angle-strong}
%\quad\text{along}\,\,.
%\end{equation}
%\begin{remark}
%It follows from \eqref{angle} that $\omega_{0}\in(-F(E_{n+1}),F(-E_{n+1}))$.
%\end{remark}
\begin{definition}\label{aniostropic-stable}
An immersion $x: \Sigma\rightarrow B$ is said to be  {\it anisotropic capillary CAMC} if it has CAMC and satisfies boundary condition \eqref{angle}. In particular, $\Sigma$ is called anisotropic free boundary CAMC hypersurface if $\omega_{0}=0$.
\end{definition}
\begin{definition}\label{strong-aniostropic-stable}
An immersion $x: \Sigma\rightarrow B$ is said to be  {\it anisotropic capillary minimal} if $H_F=0$  and $\langle\nu_{F},\bar{N}\rangle=\omega_{0}$ along $\Sigma\cap\partial B$.
\end{definition}

For  a volume-preserving admissible variation $Y$, by splitting $Y=Y^{\S}+f\nu$, we see
\begin{equation}\label{wetting-pre}
  \int_{\Sigma}f \,dA=0.
\end{equation}
Conversely, we have the following fact.
\begin{prop}[{\rm \cite[Proposition 2.1]{AS}}]\label{vol-pres-var} Let $x: \Sigma\rightarrow B$ be a compact immersed hypersurface with $\partial\Sigma\subseteq\partial B$. Then for a given $f\in C^\infty(\Sigma)$ satisfying $\int_{\Sigma} f\,dA=0$, there exists a volume-preserving admissible variation of $x$ with variational vector field having $f\nu$ as its normal part.
\end{prop}

Next we give the second variational formula for the energy functional $\mathcal{E}_{F}$. This formula is new in the anisotropic capillary case, even for $n=2$.
\begin{prop}\label{second-var}
Let $x: \Sigma\rightarrow B$ be an immersed anisotropic capillary CAMC (anisotropic capillary minimal, respectively) hypersurface and $x(\cdot, t)$ be a volume-preserving (compactly supported, respectively) admissible variation with variational vector field $Y$ having $f\nu$ as its normal part.
Then
\begin{eqnarray}\label{second}
&&\mathcal{E}_{F}''(0)=-\int_\Sigma(\div_{\Sigma}(A_{F}\nabla f)+\langle A_{F}\circ d\nu,d\nu\rangle f)f\,dA+\int_{\p \Sigma} \left(\langle A_{F}\nabla f,\mu\rangle-q_{F}f\right)f\,ds,
\end{eqnarray}
where
\begin{alignat}{2}\label{qF}
q_{F}:=\frac{1}{\langle\mu,\bar{N}\rangle^{2}}\left[\langle\mu_{F},\bar{N}\rangle h^{\partial B}(\bar{\nu},\bar{\nu})+h^{\partial B}(\bar{\nu},(\nu_F)^{\p\S})\right]-\frac{\langle\nu,\bar{N}\rangle}{\langle\mu,\bar{N}\rangle} h_{F}(\mu,\mu).
\end{alignat}
%\begin{alignat}{2}
%q_{F}:&=\frac{1}{\langle\mu,\bar{N}\rangle^{2}}(\langle\mu_{F},\bar{N}\rangle h^{\partial B}(\bar{\nu},\bar{\nu})+h^{\partial B}(\bar{\nu},(\nu_F)^{\p\S}))\label{qF}\\
%&\quad-\frac{1}{F(\nu)}\left(\frac{\omega_{0}}{\langle\mu,\bar{N}\rangle}-\langle \nu_{F},\mu\rangle\right)h_{F}(\mu,\mu),\nonumber
%\end{alignat}
%$h$ is the second fundamental form of the immersion $x$  given by $h(X, Y)= \bar g(\bar \n_X \nu, Y)$
%and $(\nu_F)^{\p\S}:=DF-\langle DF,\mu\rangle$ and $h_{F}$ is the anisotropic second fundamental form of $\Sigma$ given by \eqref{second-h}.
\end{prop}
We postpone the proof of Proposition \ref{second-var} to Appendix \ref{app}.
\begin{definition}
An anisotropic capillary CAMC immersion $x: \Sigma\rightarrow B$ is called weakly stable if $\mathcal{E}_{F}''(0)\ge 0$ for any volume-preserving admissible variation.

An anisotropic capillary minimal immersion $x: \Sigma\rightarrow B$ is called (strongly) stable if $\mathcal{E}_{F}''(0)\ge 0$ for any compactly supported admissible variation.
\end{definition}
As a consequence of Proposition \ref{second-var} and Proposition \ref{vol-pres-var}, we get the following
\begin{prop}\label{stable-domain} An anisotropic capillary CAMC immersion $x: \Sigma\rightarrow B$ is weakly stable if and only if
\begin{eqnarray}\label{stab-ineq'}
\quad-\int_\Sigma(\div_{\Sigma}(A_{F}\nabla f)+\langle A_{F}\circ d\nu,d\nu\rangle f)f\,dA+\int_{\p \Sigma} \left(\langle A_{F}\nabla f,\mu\rangle-q_{F}f\right)f\,ds\ge 0,
\end{eqnarray}
for any $f$ satisfying $\int_{\Sigma} f\, dA=0$.

An anisotropic capillary minimal immersion $x: \Sigma\rightarrow B$ is (strongly) stable if and only if \eqref{stab-ineq'} is satisfied for
 any $f\in C_{c}^{\infty}(\Sigma)$.
\end{prop}

	%\footnote{After submission of this paper we found that the

\
\section{Rigidity for stable anisotropic capillary hypersurfaces in $\mathbb{R}^{n+1}_{+}$}\label{sec4}

%\section{Some key formulae for anisotropic capillary hypersurfaces in $\mathbb{R}^{n+1}_{+}$}\label{sec4}
In this and next section, we consider the anisotropic capillary hypersurfaces in $\mathbb{R}^{n+1}_{+}$, that is the domain is $B=\rr^{n+1}_{+}$. In this setting, $h^{\p B}\equiv0$ and $\bar N=-E_{n+1}$.
Abuse of notation, we use $\mu_{n+1}$ and $\nu_{n+1}$ to denote $\langle\mu, E_{n+1}\rangle$ and $\langle\nu, E_{n+1}\rangle$ on $\Sigma$ respectively.

%The stability condition \eqref{stab-ineq'} becomes
%\begin{eqnarray}\label{stab-ineq-half}
%\quad-\int_\Sigma(\div_{\Sigma}(A_{F}\nabla f)+\langle A_{F}\circ d\nu,d\nu\rangle f)f\,dA+\int_{\p \Sigma} \left(\langle A_{F}\nabla f,\mu\rangle-q_{F}f\right)f\,ds\ge 0,
%\end{eqnarray}
%with
In this case, $q_F$ in the stability inequality \eqref{stab-ineq'} is given by
\begin{eqnarray}\label{qF-half}
q_{F}=-\frac{\nu_{n+1}}{\mu_{n+1}}h_{F}(\mu,\mu)=\frac{1}{F(\nu)}\left(\frac{\omega_{0}}{\mu_{n+1}}+\langle\nu_{F},\mu\rangle\right)h_{F}(\mu,\mu).
\end{eqnarray}
%for any function $f$ such that $\int_{\Sigma} f\, dA=0$.

The following proposition is a fundamental fact for anisotropic capillary hypersurfaces in $\mathbb{R}^{n+1}_{+}$.  %which may be familiar to experts.
 \begin{prop} \label{lemma1} Let $x: \Sigma\to \mathbb{R}^{n+1}_{+}$ be an anisotropic capillary immersion. % with $\langle\nu_{F},E_{n+1}\rangle=-\omega_{0}$.
 Then $\mu$ is an anisotropic principal direction of $\p \Sigma$ in $\Sigma$, that is, $h_{F}(e, \mu)=0$ for any $e\in T(\p \Sigma)$.
\end{prop}
\begin{proof}
For any $e\in T(\p \Sigma)$, by using \eqref{angle} and \eqref{Nbar-0}, we have
\begin{eqnarray*}
0&=&\nabla_{e}\langle\nu_{F},E_{n+1}\rangle=\langle \bar{\nabla}_{e}\nu_{F},E_{n+1}\rangle=\langle \bar{\nabla}_{e}\nu_{F},\mu_{n+1}\mu+\nu_{n+1}\nu\rangle\\
&=&\mu_{n+1}\langle \bar{\nabla}_{e}\nu_{F},\mu\rangle+\nu_{n+1}\langle \bar{\nabla}_{e}\nu_{F},\nu\rangle
=\mu_{n+1}h_{F}(e,\mu).
\end{eqnarray*}
%Here we use the fact $d\nu_{F}\in T\Sigma$.
Since $\mu_{n+1}\neq0$ along $\partial\Sigma$, we get the assertion.
\end{proof}

%The remaining part of this paper is organized as follows. In Section \ref{sec2} we review the definition and basic properties of anisotropic hypersurfaces.
%In Section \ref{sec3} we prove some key formulae for constant anisotropic mean curvature with capillary condition and obtain an anisotropic Minkowski formula in $\R^{n+1}_{+}$. In Section \ref{sec4}, we give a proof of Theorem \ref{thm0.1} after finding admissible test function \eqref{aniso-Mink0.2}.

In \cite{JWXZ}, the following anisotropic Minkowski-type formula has been proved.
\begin{prop}[{\rm \cite[Theorem 1.3]{JWXZ}}]\label{prop-integral}\,Let $x: \Sigma\to\mathbb{R}^{n+1}_{+}$ be a compact anisotropic capillary immersion. % with $\langle\nu_{F},E_{n+1}\rangle=-\omega_{0}$.
 Then
\begin{equation}\label{aniso-Mink0.2}
\int_{\Sigma} \left[n(F(\nu)+\omega_{0}\langle E^{F}_{n+1},\nu\rangle)-H_{F}\langle x,\nu\rangle\right] dA=0
\end{equation}
where %$\omega_{0}\in(-F(E_{n+1}),F(-E_{n+1}))$ and
$E_{n+1}^F$ is a constant vector field defined by
\begin{eqnarray}\label{En+1F}
E_{n+1}^F= \begin{cases}\quad\frac{\Phi\left(E_{n+1}\right)}{F\left(E_{n+1}\right)} & \text { if } \omega_0<0,\\
-\frac{\Phi\left(-E_{n+1}\right)}{F\left(-E_{n+1}\right)} & \text { if } \omega_0 \geq 0.
\end{cases}
\end{eqnarray}
\end{prop}
Note that $E_{n+1}^F$ satisfies
$\langle E_{n+1}^F,E_{n+1}\rangle=1.$
From boundary condition \eqref{angle} and the anisotropic Cauchy-Schwarz inequality (see for example \cite[Proposition 2.4]{HLMG}), we know that
\begin{equation}\label{omega0}
\omega_{0}\in(-F(E_{n+1}),F(-E_{n+1})).
\end{equation}
It was proved in \cite[Proposition 3.4]{JWXZ} that
\begin{equation}\label{Fz}
F(z)+\omega_0\left\langle E_{n+1}^F, z\right\rangle> 0 , \quad \text {for any}\,\, z \in \mathbb{S}^n.
\end{equation}
Since $\mathbb{S}^{n}$ is compact, we have
\begin{equation}\label{Fz-nonc}
0<C_{1}\leq F(z)+\omega_0\left\langle E_{n+1}^F, z\right\rangle\leq C_{2}, \quad \text {for any}\,\, z \in \mathbb{S}^n,
\end{equation}
where $C_{1}$ and $C_{2}$ are constant only depending on $\omega_{0}$.

Next we derive the equations for several geometric quantities. Denote the $F$-Jacobi operator $$J_{F}:=\div_{\Sigma}(A_{F}\nabla\cdot)+\langle A_{F}\circ d\nu,d\nu\rangle.$$ %We next derive three differential equations for geometric  quantities
%$F(\nu)$, $\langle E_{n+1}^{F},\nu\rangle$ and $\langle x,\nu\rangle$.
\begin{prop}\label{prop4.66} %Let $x: \Sigma\to\mathbb{R}^{n+1}_{+}$ be an immersion.ed hypersurface with boundary $\partial\Sigma\subseteq\partial\mathbb{R}^{n+1}_{+}$. Then
The following identities holds on $\S$:
	\begin{eqnarray}
	J_{F}F(\nu)&=&\langle \nabla H_{F},DF|_{\nu}\rangle+tr(h^{2}_{F}),\label{xmu}\\
	J_{F}\langle E_{n+1}^{F},\nu\rangle&=&\langle E^{F}_{n+1},\nabla H_{F}\rangle,\label{Ennu}\\
	J_{F}\langle x,\nu\rangle&=&\langle x,\nabla H_{F}\rangle+H_{F}.\label{Xmu1}
	\end{eqnarray}
\end{prop}
\begin{proof}
The above formulas have been shown in \cite{MX} (also see \cite{CM, Wink}). %for closed hypersurfaces.
For the convenience of reader, we give a direct computation.

For a fixed $p\in \Sigma$, let $\{e_{i}\}_{i=1}^{n}$ at $p$ be the local orthonormal basis and $\nabla_{e_{1}}e_{j}|_{p}=0$.
In the following we calculate at $p$. We have
\begin{eqnarray*}
\div_{\Sigma}(A_{F}\nabla (F(\nu)))&=&(A_{ij}F_{k}\circ\nu h_{kj})_{,i}=(A_{ij}\circ\nu)_{,i}(F_{k}\circ\nu)h_{kj}+A_{ij}((F_{k}\circ\nu)h_{kj})_{,i}
\\&=&A_{ijp}h_{pi}F_{k}h_{kj}+A_{ij}(F_{;ks}h_{si}h_{kj}+F_{k}h_{kji})\nonumber\\&=&(A_{ijp}h_{pi}h_{kj}+A_{ij}h_{kji})F_{k}+A_{ij}h_{si}h_{kj}F_{;ks}\nonumber\\
&=&(A_{ij}h_{ij})_{,k}F_{k}+A_{ij}h_{si}h_{kj}(A_{sk}-F\delta_{ks})\nonumber\\&=&\langle \nabla H_{F},DF|_{\nu}\rangle+tr(h_{F}^{2})-tr(A_{F}h^{2})F(\nu).\nonumber
\end{eqnarray*}
\begin{eqnarray*}
\div_{\Sigma}(A_{F}\nabla \langle E_{n+1}^{F},\nu\rangle)&=&(A_{ij}\nabla_{j}\langle E_{n+1}^{F},\nu\rangle)_{,i}=(A_{ij}\langle E_{n+1}^{F},\bar{\nabla}_{j}\nu\rangle)_{,i}\nonumber\\
&=&A_{ijk}h_{ki}\langle E_{n+1}^{F},\bar{\nabla}_{j}\nu\rangle+A_{ij}\langle E^{F}_{n+1},\bar{\nabla}_{i}\bar{\nabla}_{j}\nu\rangle\nonumber\\
&=&(A_{ijk}h_{ki}h_{jp}+A_{ij}h_{jpi})\langle E_{n+1}^{F},e_{p}\rangle-A_{ij}h_{jk}h_{ki}\langle E_{n+1}^{F},\nu\rangle\nonumber\\
&=&(A_{ij}h_{ij})_{,p}\langle E_{n+1}^{F},e_{p}\rangle-tr(A_{F}h^{2})\langle E_{n+1}^{F},\nu\rangle\nonumber\\
&=&\langle E_{n+1}^{F},\nabla H_{F}\rangle-tr(A_{F}h^{2})\langle E_{n+1}^{F},\nu\rangle.\nonumber
\end{eqnarray*}
%where we use Codazzi equation $h_{jpi}=h_{ijp}$ in $\mathbb{R}^{n+1}$.
%The proof for support function $\langle x,\nu\rangle$ is similar,
\begin{eqnarray*}
\div_{\Sigma}(A_{F}\nabla \langle x,\nu\rangle)&=&(A_{ij}\nabla_{j}\langle x,\nu\rangle)_{,i}=(A_{ij}\circ\nu\langle x,\bar{\nabla}_{j}\nu\rangle)_{,i}\nonumber\\
&=&A_{ijk}h_{ki}\langle x,\bar{\nabla}_{j}\nu\rangle+A_{ij}(h_{ij}+\langle x,\bar{\nabla}_{i}\bar{\nabla}_{j}\nu\rangle)\nonumber\\
&=&(A_{ijk}h_{ki}h_{jp}+A_{ij}h_{jpi})\langle x,e_{p}\rangle+H_{F}-A_{ij}h_{jk}h_{ki}\langle x,\nu\rangle\nonumber\\
&=&(A_{ij}h_{ij})_{,p}\langle x,e_{p}\rangle+H_{F}-tr(A_{F}h^{2})\langle x,\nu\rangle\nonumber\\
&=&\langle x,\nabla H_{F}\rangle+H_{F}-tr(A_{F}h^{2})\langle x,\nu\rangle.\nonumber
\end{eqnarray*}
In the above computation we have used Codazzi equation $h_{ijk}=h_{ikj}$.
\end{proof}
Next we verify the boundary equations for the geometric quantities.
\begin{prop}\label{prop-boundary} Let $x:\Sigma\to \mathbb{R}^{n+1}_{+}$ be an anisotropic capillary immersion.
Then along $\p \Sigma$, we have
\begin{eqnarray}\label{Hyp-dddd1}
&&\langle A_{F}\nabla\langle x,\nu\rangle,\mu\rangle=q_{F}\langle x,\nu\rangle,\label{boundary-222}\\
&&\langle A_{F}\nabla[F(\nu)+\omega_{0}\langle E_{n+1}^{F},\nu\rangle],\mu\rangle=q_{F}(F(\nu)+\omega_{0}\langle E_{n+1}^{F},\nu\rangle).\label{boundary-111}
\end{eqnarray}
%where $q_{F}$ is defined in \eqref{qF-half}.
\end{prop}
\begin{proof}
In this proof we always take value along $\partial \Sigma$. From Proposition \ref{lemma1} and  \eqref{Nbar-0}, we have
\begin{eqnarray*}
\langle A_{F}\nabla\langle x,\nu\rangle,\mu\rangle&=& \langle A_{F}\circ d\nu(x),\mu\rangle%=A_{ij}h_{jk}\mu^{i}\langle x,e_{k}\rangle\nonumber\\
=h_{F}(\mu,\mu)\langle x,\mu\rangle%\\&=&h_{F}(\mu,\mu)\left\langle x,\frac{1}{\mu_{n+1}}E_{n+1}-\frac{\nu_{n+1}}{\mu_{n+1}}\nu\right\rangle\nonumber\\
\\&=&-\frac{\nu_{n+1}}{\mu_{n+1}}h_{F}(\mu,\mu)\langle x,\nu\rangle=q_{F}\langle x,\nu\rangle.\nonumber
\end{eqnarray*}
Here we have used the fact $x_{n+1}=0$ on $\partial\mathbb{R}^{n+1}_{+}$.

Similarly, by Proposition \ref{lemma1} and \eqref{Nbar-0}, we have
\begin{eqnarray}
\langle A_{F}\nabla [F(\nu)+\omega_{0}\langle E_{n+1}^{F},\nu\rangle],\mu\rangle%=A_{ij}\mu^{i}\langle \bar \nabla F(\nu)+\omega_{0}E_{n+1}^{F},\bar{\nabla}_{j}\nu\rangle\\
%&&=A_{ij}h_{jk}\mu^{i}\langle DF(\nu)+\omega_{0}E_{n+1}^{F},e_{k}\rangle\nonumber\\
&=& \langle A_{F}\circ d\nu(\bar \nabla F(\nu)+\omega_{0}E_{n+1}^{F}),\mu\rangle\nonumber\\
&=&h_{F}(\mu,\mu)\langle \bar \nabla F(\nu)+\omega_{0}E_{n+1}^{F},\mu\rangle\nonumber\\
&=&h_{F}(\mu,\mu)\left\langle \bar \nabla F(\nu)+\omega_{0}E_{n+1}^{F},\frac{1}{\mu_{n+1}}E_{n+1}-\frac{\nu_{n+1}}{\mu_{n+1}}\nu\right\rangle\nonumber\\
%&&=h_{F}(\mu,\mu)\left\langle (DF(\nu)+\omega_{0}E_{n+1}^{F})^{T},\frac{1}{\mu_{n+1}}E_{n+1}\right\rangle\nonumber\\
&=&q_{F}(F(\nu)+\omega_{0}\langle E_{n+1}^{F},\nu\rangle).\nonumber
\end{eqnarray}
where in the last equality we use boundary condition $\langle\mathcal{\nu}_{F},E_{n+1}\rangle= \langle \bar \nabla F(\nu),E_{n+1}\rangle=-\omega_{0}$ and the fact $\langle E^{F}_{n+1},E_{n+1}\rangle=1$.

\end{proof}

\
Initiated by Minkowski-type formula \eqref{aniso-Mink0.2},
we define
\begin{equation}\label{test-function}
\varphi:=n(F(\nu)+\omega_{0}\langle E^{F}_{n+1},\nu\rangle)-H_{F}\langle x,\nu\rangle.
\end{equation}
\begin{prop}\label{prop-3.1}Let $x: \Sigma\to \mathbb{R}^{n+1}_{+}$ be a compact  anisotropic capillary  CAMC  immersion. %Assume $\langle\mathcal{\nu}_{F},E_{n+1}\rangle=-\omega_{0}$ along $\partial\Sigma$.
Then there hold:
	\begin{eqnarray}
J_{F}\varphi&=&n\tr(h_{F}^{2})-H_{F}^{2},\label{varphi1}\\
	\langle A_{F}\nabla\varphi,\mu\rangle&=&q_{F}\varphi,\label{Hyp-bdy1}\\
\int_{\Sigma}\varphi\,dA&=&0.\label{bdy-zero1}
%\int_{\partial M}\varphi_{n+1}\,ds&=&0.\label{bdy-zero2}
\end{eqnarray}
%Here $q_{F}$ is defined by \eqref{qF-half}.
%where $q=\frac{1}{\sin \th}\,+ \cot \th \, h(\mu, \mu).$
\end{prop}
\begin{proof}
%From Proposition \ref{Minkhorol}, we readily know that $\int_{M}\varphi_{n+1}dA=0$.
	\eqref{varphi1} follows from Proposition \ref{prop4.66} and \eqref{Hyp-bdy1} from Proposition \ref{prop-boundary}.
	\eqref{bdy-zero1} is the Minkowski-type formula \eqref{aniso-Mink0.2}.
\end{proof}

%Now we are ready to prove the rigidity result for stable anisotropic capillary hypersurfaces.
\begin{theorem}\label{thmm4.1}
	A compact, immersed  anisotropic capillary CAMC hypersurface in $\mathbb{R}^{n+1}_{+}$ is weakly stable if and only if it is a truncated Wulff shape, up to translation and homothety.
\end{theorem}
\begin{proof}
We first notice that a truncated Wulff shape in $\mathbb{R}^{n+1}_{+}$ is energy-minimizing (\cite{Wint, KP2}), hence it is weakly stable.

%In this case, $\overline{\rm Ric}\equiv -n$ and $h^{\partial B}=1$.
%The stability condition \eqref{stab-ineq} reduces to
%\begin{eqnarray}\label{stab-ineq22}
%0\leq E''(0) &=&-\int_M\vp(\De \vp+(|h|^2-n)\vp)\, dA+\int_{\p M} \vp(\n_\mu \vp-q \vp)\,ds \\
%&=& - \int_M\vp J\vp\, dA+\int_{\p M} \vp(\n_\mu \vp-q \vp)\,ds\nonumber
%\end{eqnarray}
%%with $$q=\frac{1}{\sin \th}+\cot \th \, h(\mu,\mu).$$
%for all $\vp$ such that $\int_{M} \vp\, dA=0$.
%For convenience, we omit writing the volume form $dA$ on $M$ and the area form $ds$ on $\partial M$ in an integral.
Next, we let $x: \Sigma\to\mathbb{R}^{n+1}_{+}$ be a compact weakly stable anisotropic capillary  CAMC immersion and show it is a truncated Wulff shape.
From \eqref{bdy-zero1}, we know that $\varphi$ is an admissible test function in \eqref{stab-ineq'}.
Therefore, by \eqref{Hyp-bdy1}, we have
\begin{eqnarray}\label{xeq1111}
0 &\le &-\int_\Sigma\varphi {J}_{F}\varphi\,+\int_{\p \Sigma} \varphi(\langle A_{F}\nabla\varphi,\mu\rangle-q_{F} \varphi)\,
\\&= &-\int_\Sigma [n(F(\nu)+\omega_{0}\langle E^{F}_{n+1},\nu\rangle)-H_{F}\langle x,\nu\rangle]J_{F}\varphi\,  \nonumber\\
&=&-\int_\Sigma n(F(\nu)+\omega_{0}\langle E^{F}_{n+1},\nu\rangle)J_{F}\varphi\,+H_{F}\int_{\Sigma}\langle x,\nu\rangle J_{F}\varphi.\nonumber
%&&=-\int_M (nV_{n+1}^{2}+V_{n+1}\bar{g}(E_{n+1},\nu)H)(n|h|^{2}-H^{2})\, dA+(H+n\cos\theta)\int_{M}\bar{g}(x,\nu)J\varphi_{n+1}dA.\nonumber
\end{eqnarray}
We compute the last term of \eqref{xeq1111} by Green's  formula. By \eqref{Xmu1}, \eqref{boundary-222}, \eqref{Hyp-bdy1} and \eqref{bdy-zero1}, we get
\begin{eqnarray}\label{xeq11120}
\int_{\Sigma}\langle x,\nu\rangle J_{F}\varphi&=&\int_{\Sigma}\varphi J_{F}\langle x,\nu\rangle+\int_{\partial \Sigma}\langle x,\nu\rangle\langle A_{F}\nabla\varphi,\mu\rangle-\varphi\langle A_{F}\nabla\langle x,\nu\rangle,\mu\rangle\\
&=&H_{F}\int_{\Sigma}\varphi+\int_{\partial \Sigma}\langle x,\nu\rangle (q_{F}\varphi)-\varphi (q_{F}\langle x,\nu\rangle),\nonumber\\
&=&0.\nonumber
\end{eqnarray}
%In the last equality we have used volume-preserving condition \eqref{bdy-zero1}.

Hence, inserting \eqref{varphi1} and \eqref{xeq11120} into \eqref{xeq1111}, we see
\begin{eqnarray}\label{umbilical}
\int_\Sigma (F(\nu)+\omega_{0}\langle E^{F}_{n+1},\nu\rangle)(n\tr(h_{F}^{2})-H_{F}^{2})dA\leq0.
%&&=-\int_M (nV_{n+1}^{2}+V_{n+1}\bar{g}(E_{n+1},\nu)H)(n|h|^{2}-H^{2})\, dA+(H+n\cos\theta)\int_{M}\bar{g}(x,\nu)J\varphi_{n+1}dA.\nonumber
\end{eqnarray}
By \eqref{Fz} and the fact that $n\tr(h_{F}^{2})\ge H_{F}^{2}$, we see the above inequality is in fact an equality. It follows that $n\tr(h_{F}^{2})= H_{F}^{2}$ and in turn $\S$ is anisotropic umbilical.
%\begin{equation}\label{Wulff-1}
%ntr(h_{F}^{2})-H_{F}^{2}=\sum_{1\leq i<j\leq n}(\kappa_{i}^{F}-\kappa_{j}^{F})^{2}=0,
%\end{equation}
By \cite{Pa}, $\Sigma$ is a part of the Wulff shape (up to translation and homothety).
\end{proof}

\

\section{Bernstein-type theorem for anisotropic capillary minimal surfaces}\label{sec6}
In this section we prove the following Bernstein-type theorem for properly immersed, anisotropic capillary minimal surfaces in $\mathbb{R}^{3}_{+}$.
\begin{theorem}\label{thm-Bernstein-3} Let $\Sigma$ be an immersed anisotropic capillary minimal surface in $\mathbb{R}_{+}^3$. Assume that $\Sigma$ has  Euclidean area growth, that is, there exists some $C>0$ such that
\begin{equation}\label{vol-growth-3}
\operatorname{Area}\left(\Sigma \cap B_r(0)\right)<C r^2
\end{equation}
for any $r>0$. Then $\Sigma$ is stable if and only if $\Sigma$ is a half-plane.
\end{theorem}
\begin{proof}
Assume that $\psi:=F(\nu)+\omega_{0}\langle E_{n+1}^{F},\nu\rangle$, thus from \eqref{xmu}-\eqref{Ennu} and  \eqref{boundary-111} we see that
\begin{equation}\label{equ-psi}
\begin{cases}{}
J_{F}\psi=\tr h_{F}^{2} & \text{in} \ \Sigma,\\
\langle A_{F}\nabla\psi,\mu\rangle=q_{F}\psi     & \text{on} \ \partial\Sigma,
\end{cases}
\end{equation}
For  $f \in C_c^{\infty}\left(\S\right)$, we put $\psi f$ into stability inequality \eqref{stab-ineq'}:
\begin{equation}\label{stab-nonc}
 0\leq \mathcal{E}_{F}''(0)=-\int_{\Sigma}\psi fJ_{F}(\psi f)dA+\int_{\partial\Sigma}\psi f[\langle A_{F}\nabla(\psi f),\mu\rangle-q_{F}\psi f]ds,
\end{equation}
we now calculate the first term of \eqref{stab-nonc} by integration by parts,
\begin{eqnarray}\label{equ-nonc-1}
&&-\int_{\Sigma}\psi fJ_{F}(\psi f)dA\\
&=&-\int_{\Sigma}\psi f[fJ_{F}\psi+2\langle A_{F}\nabla\psi,\nabla f\rangle+\psi\div(A_{F}\nabla f)]dA\nonumber\\
&=&-\int_{\Sigma}\psi f^{2}trh_{F}^{2}+\frac{1}{2}\langle A_{F}\nabla\psi^{2},\nabla f^{2}\rangle+\psi^{2}f\div(A_{F}\nabla f)dA\nonumber\\
&=&-\int_{\Sigma}\psi f^{2}trh_{F}^{2}+\psi^{2}f\div(A_{F}\nabla f)dA+\frac{1}{2}\int_{\Sigma}\psi^{2}\div(A_{F}\nabla f^{2})dA-\frac{1}{2}\int_{\partial\Sigma}\psi^{2}\langle A_{F}\nabla f^{2},\mu\rangle ds\nonumber\\
&=&-\int_{\Sigma}\psi f^{2}trh_{F}^{2}-\psi^{2}\langle A_{F}\nabla f,\nabla f\rangle dA-\int_{\partial\Sigma}\psi^{2}f\langle A_{F}\nabla f,\mu\rangle ds.\nonumber
\end{eqnarray}
Next we compute the boundary term of \eqref{stab-nonc},
\begin{eqnarray}\label{equ-nonc-bdy}
\int_{\partial\Sigma}\psi f[\langle A_{F}\nabla(\psi f),\mu\rangle-q_{F}\psi f]ds&=&\int_{\partial\Sigma}\psi f[f(\langle A_{F}\nabla\psi,\mu\rangle-q_{F}\psi)+\psi\langle A_{F}\nabla f,\mu\rangle]ds\\
&=&\int_{\partial\Sigma}\psi^{2} f\langle A_{F}\nabla f,\mu\rangle ds\nonumber
\end{eqnarray}
Inserting \eqref{equ-nonc-1}-\eqref{equ-nonc-bdy} into \eqref{stab-nonc}, we get
\begin{equation}\label{equ-nonc-3}
  \int_{\Sigma}\psi f^{2}trh_{F}^{2}dA\leq\int_{\Sigma}\psi^{2}\langle A_{F}\nabla f,\nabla f\rangle dA
\end{equation}
From \eqref{Fz-nonc} we obtain that
\begin{equation}\label{equ-nonc-4}
  \int_{\Sigma}f^{2}trh_{F}^{2}dA\leq C_{2}^{2}C_{1}^{-1}\int_{\Sigma}\langle A_{F}\nabla f,\nabla f\rangle dA \leq C_{2}^{2}C_{1}^{-1}\Lambda\int_{\Sigma}|\nabla f|^{2}dA.
\end{equation}
where $\Lambda$ is maximal eigenvalue of $A_{F}$ on $\mathbb{S}^{2}$.

Using the area growth condition \eqref{vol-growth} and choosing a standard logarithmic cutoff function (see for example \cite[Proposition 1.37]{ColdM}) for $f$ in \eqref{equ-nonc-4}, we deduce that $\tr h_{F}^{2}=0$. We conclude that $\Sigma$ is a half-plane in $\mathbb{R}^{3}_{+}$.

\end{proof}

\appendix
\section{Proof of the first and second variational formulas}\label{app}
The appendix is devoted to prove the first and second variational formulas, that is Proposition \ref{first-var} and Proposition \ref{second-var}.

Let $Y\in T(\p B)$ such that
\begin{equation}\label{Y-1}
  Y=Y^{\S}+f\nu= Y^{\p\S}+\langle Y,\mu\rangle\mu+f\nu.
\end{equation}
%where $Y^{\S}$ is tangent part of $\Sigma$ and $Y^{\p\S}$ is the tangent part of $Y$ to $\partial \Sigma$.
We have from \eqref{Nbar-0} that
\begin{equation*}
 0=\langle Y,\bar{N}\rangle=\langle\mu,\bar{N}\rangle\langle Y,\mu\rangle+\langle\nu,\bar{N}\rangle\langle Y,\nu\rangle\,\quad \text{on}\,\,\partial\Sigma,
\end{equation*}
Therefore,
\begin{equation}\label{Y-mu}
  \langle Y,\mu\rangle=-\frac{\langle\nu,\bar{N}\rangle}{\langle\mu,\bar{N}\rangle}f\, \quad \text{on}\,\,\partial \Sigma.
\end{equation}

On the other hand, from \eqref{Y-1}, \eqref{Y-mu} and \eqref{nubar-0}, we see $Y$ can be also expressed as follows
\begin{equation}\label{Y-2}
  Y=Y^{\p\S}-\frac{f}{\langle\mu,\bar{N}\rangle}(\langle\nu,\bar{N}\rangle\mu-\langle\mu,\bar{N}\rangle\nu)=Y^{\p\S}+\frac{f}{\langle\mu,\bar{N}\rangle}\bar{\nu}.
\end{equation}
It follows that,
\begin{equation}\label{Y-mubar}
  \langle Y,\bar{\nu}\rangle=\frac{1}{\langle\mu,\bar{N}\rangle}f\, \quad \text{on}\,\,\partial \Sigma.
\end{equation}

We use a prime to denote the time derivative at $t=0$ in the following.
\begin{lemma}\label{N'-and-v'}\
Let $\nabla^{\p\S}$ denote the gradient on $\partial \Sigma$. Let $S^\S$, $S^{\p B}$ denote the shape operator of $\Sigma$ in $B$ with respect to $\nu$ and $\p B$ in $B$ with respect to $\bar N$ respectively. Let $S_{1}$ and $S_{2}$ denote the shape operator of $\partial \Sigma$ in $\Sigma$ with respect to $\mu$, and of $\partial \Sigma$ in $\partial B$ with respect to $\bar{\nu}$ respectively. Then we have
\begin{eqnarray}
\nu'&=&S^\S(Y^{\S})-\nabla f,\label{nu'}\\
\mu'&=&(-h(Y^{\S},\mu)+\nabla_{\mu}f)\nu+S_{1}(Y^{\p\S})+\frac{\langle\nu,\bar{N}\rangle}{\langle\mu,\bar{N}\rangle}\,\nabla^{\p\S}f\label{mu'}
\\&&+\frac{\langle\nu,\bar{N}\rangle^{2}}{\langle\mu,\bar{N}\rangle^{2}}f\left[S^\S(\mu)-h(\mu,\mu)\mu\right]+\frac{1}{\langle\mu,\bar{N}\rangle^{2}}f\left[S^{\partial B}(\bar{\nu})-h^{\partial B}(\bar{\nu},\bar{\nu})\bar{\nu}\right],\nonumber\\
\bar{\nu}'&=&S_{2}(Y^{\p\S})-\frac{1}{\langle\mu,\bar{N}\rangle}\nabla^{\p\S}f-h^{\partial B}(Y,\bar{\nu})\bar{N}\label{barnu'}\\&&-\frac{\langle\nu,\bar{N}\rangle}{\langle\mu,\bar{N}\rangle^{2}}f\left[S^\S(\mu)-h(\mu,\mu)\mu
    +S^{\partial B}(\bar{\nu})-h^{\partial B}(\bar{\nu},\bar{\nu})\bar{\nu}\right].\nonumber
\end{eqnarray}
\end{lemma}
\begin{proof}
Let $\{e_{i}\}_{i=1}^{n}$ be an orthonormal basis of $T_{p}\Sigma$ for some $p\in \Sigma$ and denote $e_{i}(t)=(x(t,\cdot))_{\ast}(e_{i})$. Using the fact $\langle e_{i}(t),\nu(t)\rangle=0$ and $[e_{i}(t), Y(t)]=0$, we have
\begin{alignat}{2}
\nu'&=\sum_{i=1}^{n}\langle \nu', e_{i}\rangle e_{i}=-\sum_{i=1}^{n}\langle\nu, e_{i}'\rangle e_{i}\nonumber\\
&=-\sum_{i=1}^{n}\langle\nu, \bar{\nabla}_{e_{i}}Y\rangle e_{i}=-\sum_{i=1}^{n}\langle\nu,\bar{\nabla}_{e_{i}}(Y^{\S}+f\nu) \rangle e_{i}\nonumber\\
&=\sum_{i=1}^{n}\langle S^\S(Y^{\S}), e_{i}\rangle e_{i}-\sum_{i=1}^{n}df(e_{i})e_{i}\nonumber\\
&=S^\S(Y^{\S})-\nabla f.\nonumber
\end{alignat}
This is \eqref{nu'}.
As a consequence of \eqref{nu'} we get
\begin{equation}\label{nu-n}
  \langle\mu',\nu\rangle=-\langle\mu,\nu'\rangle=-h(Y^{\S}, \mu)+\nabla_{\mu}f.
\end{equation}
Let now $\{\tau_\a\}_{\a=1}^{n-1}$ be an orthonormal basis of $T_p(\partial \Sigma)$ and denote $\tau_\a(t)=(x(t,\cdot))_{\ast}(\tau_\a)$. Using \eqref{Y-1}, \eqref{Y-mu} and the fact $[\tau_\a(t), Y(t)]=0$, we have

\begin{eqnarray}\label{nu-n2}
\langle\mu', \tau_\a\rangle&=&-\langle\mu,\tau'_{\a}\rangle=-\langle\mu,\bar \n_{\tau_\a}Y\rangle=-\left\langle\mu, \bar \n_{\tau_\a}\left(Y^{\p\S}-\frac{\langle\nu,\bar{N}\rangle}{\langle\mu,\bar{N}\rangle}f\mu+f\nu\right)\right\rangle\\
&=&-\langle\mu, \bar \n_{\tau_\a}Y^{\p\S}\rangle+d\left(\frac{\langle\nu,\bar{N}\rangle}{\langle\mu,\bar{N}\rangle}f\right)(\tau_\a)-f\langle\mu, \bar \n_{\tau_\a}\nu\rangle\nonumber\\
&=&-\langle\mu, \bar \n_{\tau_\a}Y^{\p\S}\rangle+\left(\langle\mu,\bar{N}\rangle d\langle\nu,\bar{N}\rangle(\tau_\a)-\langle\nu,\bar{N}\rangle d\langle\mu,\bar{N}\rangle(\tau_\a)\right)\frac{f}{\langle\mu,\bar{N}\rangle^{2}}\nonumber\\
&&+\frac{\langle\nu,\bar{N}\rangle}{\langle\mu,\bar{N}\rangle}df(\tau_\a)-f h(\mu,\tau_\a).\nonumber
\end{eqnarray}
From \eqref{nubar-0}, we see
\begin{eqnarray}\label{xeq1}
&&\langle\mu,\bar{N}\rangle d\langle\nu,\bar{N}\rangle(\tau_\a)-\langle\nu,\bar{N}\rangle d\langle\mu,\bar{N}\rangle(\tau_\a)
\\&=&\langle\mu,\bar{N}\rangle \langle\bar \n_{\tau_\a}\nu,\bar{N}\rangle -\langle\nu,\bar{N}\rangle \langle\bar \n_{\tau_\a}\mu,\bar{N}\rangle +d\bar N(\bar \nu, \tau_\a)\nonumber
\\&=&\langle\mu,\bar{N}\rangle^2 \langle\bar \n_{\tau_\a}\nu,\mu\rangle-\langle\nu,\bar{N}\rangle^2 \langle\bar \n_{\tau_\a}\mu,\nu\rangle+ h^{\p B}(\tau_\a, \bar \nu)\nonumber\\&=& h(\mu,\tau_\a)+h^{\partial B}(\bar{\nu},\tau_\a).\nonumber
\end{eqnarray}
Replacing \eqref{xeq1} in \eqref{nu-n2}, we get
\begin{eqnarray}\label{nu-n2'}
&&\langle\mu', \tau_\a\rangle%&=&-\langle\mu, \bar \n_{\tau_\a}Y^{\p\S}\rangle+\frac{1}{\langle\mu,\bar{N}\rangle^{2}} (h^{\partial B}(\bar{\nu},\tau_\a)+h(\mu,\tau_\a))f +\frac{\langle\nu,\bar{N}\rangle}{\langle\mu,\bar{N}\rangle}df(\tau_\a)-h(\mu,\tau_\a)f\nonumber\\
=-\langle\mu, \bar \n_{\tau_\a}Y^{\p\S}\rangle+\frac{\langle\nu,\bar{N}\rangle}{\langle\mu,\bar{N}\rangle}df(\tau_\a)+\frac{\langle\nu,\bar{N}\rangle^{2}}{\langle\mu,\bar{N}\rangle^{2}}h(\mu,\tau_\a)f+\frac{1}{\langle\mu,\bar{N}\rangle^{2}} h^{\partial B}(\bar{\nu},\tau_\a)f.
\end{eqnarray}
It follows from \eqref{nu-n} and \eqref{nu-n2'} that
\begin{alignat}{2}
\mu'&=\langle\mu',\mu\rangle\mu+\langle\mu', \nu\rangle\nu+\sum_{\a=1}^{n-1}\langle\mu', \tau_\a\rangle \tau_\a\label{second-term}\\
&=(-h(Y^{\S}, \mu)+\nabla_{\mu}f)\nu+S_{1}(Y^{\p\S})+\frac{\langle\nu,\bar{N}\rangle}{\langle\mu,\bar{N}\rangle}\,\nabla^{\p\S}f\nonumber\\
&\quad+\frac{\langle\nu,\bar{N}\rangle^{2}}{\langle\mu,\bar{N}\rangle^{2}}f\left[S^\S(\mu)-h(\mu,\mu)\mu\right]+\frac{1}{\langle\mu,\bar{N}\rangle^{2}}f\left[S^{\partial B}(\bar{\nu})-h^{\partial B}(\bar{\nu},\bar{\nu})\bar{\nu}\right].\nonumber
\end{alignat}
This is  \eqref{mu'}.

Lastly, using $[\tau_\a(t), Y(t)]=0$ again and (\ref{Y-2}), we have
\begin{alignat}{2}
\langle\bar{\nu}',\tau_\a\rangle&=-\langle\bar{\nu},\tau'_{\a}\rangle=-\langle\bar{\nu},\bar \n_{\tau_\a}Y\rangle\label{third-term}\\
&=-\langle\bar{\nu},\bar \n_{\tau_\a}Y^{\p\S}\rangle-d\left(\frac{f}{\langle\mu,\bar{N}\rangle}\right)(\tau_\a)\nonumber\\
&=\langle S_{2}(Y^{\p\S}),\tau_\a\rangle-\frac{1}{\langle\mu,\bar{N}\rangle}df(\tau_\a)-\frac{\langle\nu,\bar{N}\rangle}{\langle\mu,\bar{N}\rangle^{2}}f\left[h(\mu,\tau_\a)+h^{\partial B}(\bar{\nu},\tau_\a)\right].\nonumber
\end{alignat}
Here the last equality we used \eqref{mu-0} and \eqref{Nbar-0}.

Now \eqref{barnu'}  follows from (\ref{third-term}) and the fact $\langle\bar{\nu}',\bar{N}\rangle=-h^{\partial B}(Y,\bar{\nu}).$
\end{proof}

\begin{lemma} Along $\partial \Sigma$, we have
\begin{eqnarray}\label{N-and-n}
&&S_{1}(Y^{\p\S},Y^{\p\S})+\frac{\langle\nu,\bar{N}\rangle}{\langle\mu,\bar{N}\rangle}\<\nabla f,Y^{\p\S}\>+\frac{\langle\nu,\bar{N}\rangle^{2}}{\langle\mu,\bar{N}\rangle^{2}}f\left[h(Y^{\p\S},\mu)+h^{\partial B}(Y^{\p\S},\bar{\nu})\right]\\&=&-\langle\nu,\bar{N}\rangle\langle Y,\bar{\nu}'\rangle+\langle\mu,\bar{N}\rangle h^{\partial B}(Y^{\p\S},Y^{\p\S})\nonumber
\end{eqnarray}
\end{lemma}
\begin{proof}
From  \eqref{barnu'}  and  \eqref{Y-2}, we have
\begin{alignat}{2}
&-\langle\nu,\bar{N}\rangle\langle Y,\bar{\nu}'\rangle\\
%&=-\langle\nu,\bar{N}\rangle\left\langle Y,S_{2}(Y^{\p\S})-\frac{1}{\langle\mu,\bar{N}\rangle}\nabla^{\p\S}f-\frac{\langle\nu,\bar{N}\rangle}{\langle\mu,\bar{N}\rangle^{2}}f\sum_\a\left[h(\mu,\tau_\a)+h^{\partial B}(\bar{\nu},\tau_\a)\right]\tau_\a\right\rangle\nonumber\\
&=\langle\nu,\bar{N}\rangle\langle \bar{\nabla}_{Y^{\p\S}}Y^{\p\S},\bar{\nu}\rangle+\frac{\langle\nu,\bar{N}\rangle}{\langle\mu,\bar{N}\rangle}\langle Y^{\p\S},\nabla^{\p\S}f\rangle+\frac{\langle\nu,\bar{N}\rangle^{2}}{\langle\mu,\bar{N}\rangle^{2}}f\left[h(Y^{\p\S},\mu)+h^{\partial B}(Y^{\p\S},\bar{\nu})\right]\nonumber
\end{alignat}
Using \eqref{mu-0}, we see
\begin{alignat}{2}
\langle\nu,\bar{N}\rangle\langle \bar{\nabla}_{Y^{\p\S}}Y^{\p\S},\bar{\nu}\rangle
&=\langle \bar{\nabla}_{Y^{\p\S}}Y^{\p\S},-\mu+\langle\mu,\bar{N}\rangle \bar{N}\rangle
&=S_{1}(Y^{\p\S},Y^{\p\S})-\langle\mu,\bar{N}\rangle h^{\partial B}(Y^{\p\S},Y^{\p\S}).\nonumber
\end{alignat}
The assertion follows.
\end{proof}

\begin{lemma}\label{N-and-n-2} Along $\partial \Sigma$, we have
\begin{equation}\label{xeq6}
h_{F}(Y^{\S},\mu)+\frac{1}{\langle\mu,\bar{N}\rangle}h^{\partial B}(Y,(\nu_F)^{\p\S})+\frac{\langle\mu_{F},\bar{N}\rangle}{\langle\mu,\bar{N}\rangle}h^{\partial B}(Y,\bar{\nu})=q_{F}f.
\end{equation}
%where $q_{F}$ is defined by \eqref{qF} and $(\nu_F)^{\p\S}=DF-\langle DF,\mu\rangle\mu$ on $\partial\Sigma$.
\end{lemma}
\begin{proof}

Using the capillary condition $\langle\nu_{F},\bar{N}\rangle=\omega_{0}$,  \eqref{Nbar-0} and the fact that $\langle\bar{\nabla}_{Y^{\p\S}}\nu_{F}, \nu\rangle=0$, we calculate that
\begin{alignat}{2}
\langle\mu,\bar{N}\rangle h_{F}(Y^{\p\S},\mu)&=\langle\mu,\bar{N}\rangle\langle\bar{\nabla}_{Y^{\p\S}}\nu_{F},\mu\rangle=\langle\bar{\nabla}_{Y^{\p\S}}\nu_{F},\bar{N}\rangle=-\langle\nu_{F},\bar{\nabla}_{Y^{\p\S}}\bar{N}\rangle\label{qf-9}\\
%&=-\langle\nu_{F},h^{\partial B}(Y^{\p\S},\bar{\nu})\bar{\nu}+h^{\partial B}(Y^{\p\S},\tau_\a)\tau_\a\rangle\\
&=-\langle\nu_{F},\bar{\nu}\rangle h^{\partial B}(Y^{\p\S},\bar{\nu})-h^{\partial B}(Y^{\p\S},(\nu_F)^{\p\S})\nonumber\\
&=-\langle\mu_{F},\bar{N}\rangle h^{\partial B}(Y^{\p\S},\bar{\nu})-h^{\partial B}(Y^{\p\S},(\nu_F)^{\p\S}).\nonumber
\end{alignat}
Here in the last equality we used \eqref{xeq3}.

From \eqref{Y-1}, \eqref{Y-mu}, \eqref{Y-2} and \eqref{qf-9}, we obtain that
\begin{alignat}{2}
&h_{F}(Y^{\S},\mu)+\frac{1}{\langle\mu,\bar{N}\rangle}h^{\partial B}(Y,(\nu_F)^{\p\S})+\frac{\langle\mu_{F},\bar{N}\rangle}{\langle\mu,\bar{N}\rangle}h^{\partial B}(Y,\bar{\nu})\\
&=h_{F}(Y^{\p\S},\mu)+\langle Y,\mu\rangle h_{F}(\mu,\mu)+\frac{1}{\langle\mu,\bar{N}\rangle}h^{\partial B}(Y,(\nu_F)^{\p\S})+\frac{\langle\mu_{F},\bar{N}\rangle}{\langle\mu,\bar{N}\rangle}h^{\partial B}(Y,\bar{\nu})\nonumber\\
%&=\langle Y,\mu\rangle h_{F}(\mu,\mu)
%+\frac{1}{\langle\mu,\bar{N}\rangle}\left(\langle\mu,\bar{N}\rangle h_{F}(Y^{\p\S},\mu)+\sum_{\alpha=1}^{n-1}h^{\partial B}(Y,\tau_\a)\langle DF,\tau_\a\rangle+\langle\mu_{F},\bar{N}\rangle h^{\partial B}(Y,\bar{\nu})\right)\nonumber\\
&=-\frac{\langle\nu,\bar{N}\rangle}{\langle\mu,\bar{N}\rangle}fh_{F}(\mu,\mu)
+\frac{1}{\langle\mu,\bar{N}\rangle}\left(h^{\partial B}(Y-Y^{\p\S},(\nu_F)^{\p\S})+\langle\mu_{F},\bar{N}\rangle h^{\partial B}(Y-Y^{\p\S},\bar{\nu})\right)\nonumber\\
&=-\frac{\langle\nu,\bar{N}\rangle}{\langle\mu,\bar{N}\rangle}fh_{F}(\mu,\mu)
+\frac{1}{\langle\mu,\bar{N}\rangle^{2}}\left(h^{\partial B}(\bar \nu,(\nu_F)^{\p\S})+\langle\mu_{F},\bar{N}\rangle h^{\partial B}(\bar{\nu},\bar{\nu})\right)f\nonumber\\
&=q_{F}f.\nonumber
\end{alignat}

\end{proof}

\noindent{\bf Proof of Proposition \ref{first-var}.}\,\,For an admissible variation $Y$, we calculate the first variation of anisotropic energy $\mathcal{E}_{F}$ as follows
\begin{alignat}{2}
\mathcal{E}_{F}'(0)&=\mathcal{A}_{F}'(0)+\omega_{0} \mathcal{A}_{W}'(0)=\int_{\Sigma}\frac{\p}{\p t}\Big|_{t=0}F(\nu)dA+\int_{\Sigma}F(\nu)\frac{\p}{\p t}\Big|_{t=0}dA_{t}+\omega_{0}\int_{\partial\Sigma}\langle\bar{\nu},Y\rangle ds\nonumber\\
&=\int_{\Sigma}\langle \bar \n F(\nu),\nu'\rangle+F(\nu)\div_{\Sigma}YdA+\omega_{0}\int_{\partial\Sigma}\langle \bar{\nu},Y\rangle ds\nonumber\\
&=\int_{\Sigma}\langle DF(\nu),-\nabla f+S^\S(Y^{\S})\rangle+F(\nu)(\div_{\Sigma}Y^{\S}+fH)dA+\omega_{0}\int_{\partial\Sigma}\langle \bar{\nu},Y\rangle ds\nonumber\\
&=\int_{\Sigma}(\div_{\Sigma}DF+F(\nu)H)fdA+\int_{\partial\Sigma}F(\nu)\langle Y,\mu\rangle-f\langle DF,\mu\rangle ds+\omega_{0}\int_{\partial\Sigma}\langle\bar{\nu},Y\rangle ds\nonumber\\
&=\int_{\Sigma}H_{F}fdA+\int_{\partial\Sigma}\langle F(\nu)\mu-\langle DF,\mu\rangle\nu+\omega_{0}\bar{\nu},Y\rangle ds\nonumber\\
&=\int_{\Sigma}H_{F}fdA+\int_{\partial\Sigma}\langle \mu_{F}+\omega_{0}\bar{\nu},Y\rangle ds.\nonumber
\end{alignat}
%where the above equality we used \eqref{HF} and \eqref{muF}.
\qed

%Next we compute the second variation of the anisotropic energy $\mathcal{E}$. The computation is very close to the one by Ros-Souam \cite{RS}. \\

\noindent{\bf Proof of Proposition \ref{second-var}.}
Let $x:\S\to B$ be an anisotropic capillary CAMC (anisotropic capillary minimal, respectively) immersion, that is, $H_{F}=const$ ($H_{F}=0$, respectively) in $\Sigma$ and $\langle\nu_{F},\bar{N}\rangle=\omega_{0}$ on $\partial\Sigma$ and $x(t,\cdot)$ be a volume-preserving (compactly supported, respectively)  admissible variation. It is direct to see that the first variational formula is true for any $t\in (-\ep,\ep)$,  that is,
\begin{alignat}{2}
\mathcal{E}_{F}'(t)&=\int_{\Sigma_t}H_{F}fdA+\int_{\partial\Sigma_t}\langle \mu_{F}+\omega_{0}\bar{\nu},Y\rangle ds.\nonumber
\end{alignat}
Thus we have
\begin{alignat*}{2}
\mathcal{E}_{F}''(0)&=\int_{\Sigma}H_{F}'fdA+H_{F}\left(\int_{\Sigma}f dA\right)'\\
&\quad+\int_{\partial \Sigma}\langle Y', \mu_{F}+\omega_{0}\bar{\nu}\rangle+\langle Y, \mu_{F}'+\omega_{0}\bar{\nu}'\rangle ds+\int_{\partial\Sigma}\langle Y,\mu_{F}+\omega_{0}\bar{\nu}\rangle\frac{\p}{\p t}\Big|_{t=0}ds_{t}.
\end{alignat*}
Observe that in the case of CAMC with volume preserving variation, we have
\begin{equation}\label{volume-p}
\left(\int_{\Sigma}f dA\right)'=\mathcal{V}''(0)=0,
\end{equation}
and in the case of anisotropic minimal,  $H_F=0$. Hence in both cases, the term $H_{F}\left(\int_{\Sigma}f dA\right)'=0$.

Also, $\langle Y,\mu_{F}+\omega_{0}\bar{\nu}\rangle=0$ along $\p\S$ since $Y|_{\p\S}\in T(\partial B)$. Moreover, we have the evolution equation  (see \cite{KP})
\begin{equation*}
  H_{F}'=-(\div_{\Sigma}(A_{F}\nabla f)+\langle A_{F}\circ d\nu,d\nu\rangle f).
\end{equation*}
It follows that
\begin{alignat}{2}\label{variation-formula}
\mathcal{E}_{F}''(0)=-\int_{\Sigma}(\div_{\Sigma}(A_{F}\nabla f)+\langle A_{F}\circ d\nu,d\nu\rangle f) fdA+\int_{\partial \Sigma}\langle Y', \mu_{F}+\omega_{0}\bar{\nu}\rangle+\langle Y, \mu_{F}'+\omega_{0}\bar{\nu}'\rangle ds.
\end{alignat}
So to prove the formula for $\mathcal{E}_{F}''(0)$ we only need to compute the boundary term
\begin{equation}\label{boundary-term}
\int_{\partial \Sigma}\langle Y', \mu_{F}+\omega_{0}\bar{\nu}\rangle ds+\int_{\partial \Sigma}\langle Y, \mu_{F}'+\omega_{0}\bar{\nu}'\rangle ds.
\end{equation}
We now calculate the first term of \eqref{boundary-term}. Since $\mu_{F}+\omega_{0}\bar{\nu}$ is parallel to $\bar{N}$ along $\partial\Sigma$, we have
\begin{alignat}{2}
\langle Y',\mu_{F}+\omega_{0}\bar{\nu}\rangle&=\langle \bar{\nabla}_{Y}Y,\mu_{F}+\omega_{0}\bar{\nu}\rangle=-\langle\mu_{F},\bar{N}\rangle h^{\partial B}(Y,Y). \label{Y'}
\end{alignat}

Next we calculate the second term of \eqref{boundary-term}.
According to the definition of $\mu_{F}$, by using \eqref{euc-sph-der1} and \eqref{euc-sph-der2}, we see that
\begin{alignat}{2}\label{muFdi}
\mu_{F}'&=(F(\nu)\mu-\langle \nu_F,\mu\rangle\nu)'\\
&=F(\nu)'\mu+F(\nu)\mu'-\langle \bar{\nabla}F(\nu),\mu\rangle'\nu-\langle D F(\nu),\mu\rangle\nu'\nonumber\\
&=\langle \bar{\nabla} F(\nu),\nu'\rangle\mu+F(\nu)\mu'-\left[\bar{\nabla}^{2}F(\nu)(\nu',\mu)+\langle \bar{\nabla}F(\nu),\mu'\rangle \right]\nu-\langle DF(\nu),\mu\rangle\nu'\nonumber\\
&=\langle DF(\nu),\nu'\rangle\mu+F(\nu)\mu'-\left[(D^{2}F(\nu)+F(\nu){\rm Id})(\nu',\mu)+\langle DF(\nu)+F(\nu)\nu,\mu'\rangle\right]\nu-\langle DF(\nu),\mu\rangle\nu'\nonumber\\
&=\langle DF(\nu),\nu'\rangle\mu+F(\nu)\mu'-\left[D^{2}F(\nu)(\nu',\mu)+\langle DF(\nu),\mu'\rangle \right]\nu-\langle DF(\nu),\mu\rangle\nu'.\nonumber
\end{alignat}
We remind here that $\bar{\nabla} F$ is the Euclidean covariant derivative on the one-homogenous extension of $F$, and  $DF$ is the covariant derivative of $F$ with respect to $\ss^n$.

It follows that
\begin{alignat}{2}\label{muFadd}
\langle Y,\mu_{F}'+\omega_{0}\bar{\nu}'\rangle&=-D^{2}F(\nu)(\nu',\mu) f+F(\nu)\langle Y,\mu'\rangle+\omega_{0}\langle Y,\bar{\nu}'\rangle
\nonumber\\
&\quad+\langle DF(\nu),\nu'\rangle\langle Y,\mu\rangle-\langle DF(\nu),\mu'\rangle f-\langle DF(\nu),\mu\rangle\langle Y,\nu'\rangle\nonumber\\
&:=\uppercase\expandafter{\romannumeral1}+\uppercase\expandafter{\romannumeral2}+\uppercase\expandafter{\romannumeral3}+\uppercase\expandafter{\romannumeral4}+\uppercase\expandafter{\romannumeral5}+\uppercase\expandafter{\romannumeral6}.
\end{alignat}
We now tackle the above terms one by one. Using \eqref{nu'}, we have
\begin{alignat}{2}\label{muFadd-5}
\uppercase\expandafter{\romannumeral1}&=-D^{2}F(\nu)(\nu',\mu) f\\
&=-\langle(D^{2}F(\nu)+F(\nu){\rm Id})\mu,\nu'\rangle f+F(\nu)\langle\mu,\nu'\rangle f\nonumber\\
&=-\langle A_{F}(\nu)\cdot\mu,-\nabla f+S^\S(Y^{\S})\rangle f+F(\nu)\langle\mu,\nu'\rangle f\nonumber\\
&=\langle A_{F}(\nu)\nabla f,\mu\rangle f-h_{F}(Y^{\S},\mu)f-F(\nu)\langle\mu',\nu\rangle f.\nonumber
\end{alignat}

Utilizing \eqref{mu'}, \eqref{N-and-n} and \eqref{Y-2}, we get
\begin{alignat}{2}\label{muFadd-3}
\langle Y,\mu'\rangle%&=\left\langle Y,    \langle\mu',\nu\rangle\nu+S_{1}(Y^{\p\S})+\frac{\langle\nu,\bar{N}\rangle}{\langle\mu,\bar{N}\rangle}\nabla^{\p\S}f\right\rangle\\
%&\quad+\left\langle Y,\frac{\langle\nu,\bar{N}\rangle^{2}}{\langle\mu,\bar{N}\rangle^{2}}fh(\mu,\tau_\a)\tau_\a+\frac{1}{\langle\mu,\bar{N}\rangle^{2}}fh^{\partial B}(\bar{\nu},\tau_\a)\tau_\a
%\right\rangle\nonumber\\
&=\langle\mu',\nu\rangle f+S_{1}(Y^{\p\S},Y^{\p\S})+\frac{\langle\nu,\bar{N}\rangle}{\langle\mu,\bar{N}\rangle}\<\nabla f, Y^{\p\S}\>\\
&\quad+\frac{\langle\nu,\bar{N}\rangle^{2}}{\langle\mu,\bar{N}\rangle^{2}}fh(Y^{\p\S},\mu)+\left(1+\frac{\langle\nu,\bar{N}\rangle^{2}}{\langle\mu,\bar{N}\rangle^{2}}\right)h^{\partial B}(Y^{\p\S},\bar{\nu})f\nonumber\\
&=\langle\mu',\nu\rangle f-\langle\nu,\bar{N}\rangle\langle Y,\bar{\nu}'\rangle+\langle\mu,\bar{N}\rangle h^{\partial B}(Y^{\p\S},Y^{\p\S})+h^{\partial B}(Y^{\p\S},\bar{\nu})f\nonumber\\
&=\langle\mu',\nu\rangle f-\langle\nu,\bar{N}\rangle\langle Y,\bar{\nu}'\rangle+\langle\mu,\bar{N}\rangle h^{\partial B}(Y^{\p\S},Y).\nonumber
\end{alignat}
Note that $\bar{\nu}=\frac{1}{\langle\mu,\bar{N}\rangle}\nu-\frac{\langle\nu,\bar{N}\rangle}{\langle\mu,\bar{N}\rangle}\bar{N}$. It follows that
\begin{eqnarray}\label{xeq4}
S_{2}(Y^{\p\S},Y^{\p\S})=\frac{1}{\langle\mu,\bar{N}\rangle}h(Y^{\p\S},Y^{\p\S})-\frac{\langle\nu,\bar{N}\rangle}{\langle\mu,\bar{N}\rangle}h^{\partial B}(Y^{\p\S},Y^{\p\S}).
\end{eqnarray}
%Here the second equality we used $1=\langle\mu,\bar{N}\rangle^{2}+\langle\nu,\bar{N}\rangle^{2}$ on $\partial\Sigma$.
Applying \eqref{barnu'}, \eqref{xeq4} and \eqref{Y-2}, we have
\begin{alignat}{2}\label{muFadd-39}
\langle Y,\bar{\nu}'\rangle%&=\left\langle Y,S_{2}(Y^{\p\S})-\frac{1}{\langle\mu,\bar{N}\rangle}\nabla^{\p\S}f-\frac{\langle\nu,\bar{N}\rangle}{\langle\mu,\bar{N}\rangle^{2}}f(h(\mu,\tau_\a)+h^{\partial B}(\bar{\nu},\tau_\a))\tau_\a\right\rangle\\
&=S_{2}(Y^{\p\S},Y^{\p\S})-\frac{1}{\langle\mu,\bar{N}\rangle}\<\nabla f, Y^{\p\S}\>-\frac{\langle\nu,\bar{N}\rangle}{\langle\mu,\bar{N}\rangle^{2}}f\left[h(Y^{\p\S},\mu)+h^{\partial B}(Y^{\p\S},\bar{\nu})\right]\\
%&=\frac{1}{\langle\mu,\bar{N}\rangle}h(Y^{\p\S},Y^{\p\S})-\frac{\langle\nu,\bar{N}\rangle}{\langle\mu,\bar{N}\rangle}h^{\partial B}(Y^{\p\S},Y^{\p\S})-\frac{1}{\langle\mu,\bar{N}\rangle}\<\nabla f, Y^{\p\S}\>\nonumber\\
%&\quad-\frac{\langle\nu,\bar{N}\rangle}{\langle\mu,\bar{N}\rangle^{2}}f(h(Y^{\p\S},\mu)+h^{\partial B}(Y^{\p\S},\bar{\nu}))\nonumber\\
&=\frac{1}{\langle\mu,\bar{N}\rangle}h(Y^{\p\S},Y^{\p\S})-\frac{1}{\langle\mu,\bar{N}\rangle}\<\nabla f, Y^{\p\S}\>\nonumber\\&\quad-\frac{\langle\nu,\bar{N}\rangle}{\langle\mu,\bar{N}\rangle^{2}}fh(Y^{\p\S},\mu)
-\frac{\langle\nu,\bar{N}\rangle}{\langle\mu,\bar{N}\rangle}h^{\partial B}(Y^{\p\S},Y).\nonumber
%&=\left(-\frac{1}{\mu_{n+1}}h(Y^{\p\S},Y^{\p\S})+\frac{1}{\mu_{n+1}}\<\nabla f, Y^{\p\S}\>+\frac{\nu_{n+1}}{\mu_{n+1}^{2}}fh(Y^{\p\S},\mu)\right)\nonumber
\end{alignat}
Combining \eqref{muFadd-3} with \eqref{muFadd-39}, by using the capillary condition \eqref{angle}, we obtain that
\begin{alignat}{2}\label{muFadd-4}
\uppercase\expandafter{\romannumeral2}+\uppercase\expandafter{\romannumeral3}&=F(\nu)\langle Y,\mu'\rangle+\omega_{0}\langle Y,\bar{\nu}'\rangle\\
&=F(\nu)\langle\mu',\nu\rangle f+(-F(\nu)\langle\nu,\bar{N}\rangle+\omega_{0})\langle Y,\bar{\nu}'\rangle+F(\nu)\langle\mu,\bar{N}\rangle h^{\partial B}(Y^{\p\S},Y)\nonumber\\
&=F(\nu)\langle\mu',\nu\rangle f+\langle\mu,\bar{N}\rangle\langle \nu_F,\mu\rangle\langle Y,\bar{\nu}'\rangle+F(\nu)\langle\mu,\bar{N}\rangle h^{\partial B}(Y^{\p\S},Y)\nonumber\\
%&=F(\nu)\langle\mu',\nu\rangle f+F(\nu)\langle\mu,\bar{N}\rangle h^{\partial B}(Y^{\p\S},Y)\nonumber\\
%&\quad+\langle \nu_F,\mu\rangle\left(h(Y^{\p\S},Y^{\p\S})-\<\nabla f, Y^{\p\S}\>-\frac{\langle\nu,\bar{N}\rangle}{\langle\mu,\bar{N}\rangle}fh(Y^{\p\S},\mu)-\langle\nu,\bar{N}\rangle h^{\partial B}(Y^{\p\S},Y)\right)\nonumber\\
&=F(\nu)\langle\mu',\nu\rangle f
+\langle\mu_{F},\bar{N}\rangle h^{\partial B}(Y^{\p\S},Y)\nonumber\\
&\quad+\langle \nu_F,\mu\rangle\left(h(Y^{\p\S},Y^{\p\S})-\<\nabla f, Y^{\p\S}\>-\frac{\langle\nu,\bar{N}\rangle}{\langle\mu,\bar{N}\rangle}fh(Y^{\p\S},\mu)\right).\nonumber
%\left(-h(Y^{\p\S},Y^{\p\S})+\<\nabla f, Y^{\p\S}\>+\frac{\nu_{n+1}}{\mu_{n+1}}fh(Y^{\p\S},\mu)\right)\nonumber
\end{alignat}
%Here the third equality we used capillary condition $F(\nu)\langle\nu,\bar{N}\rangle+\langle\mu,\bar{N}\rangle\langle DF,\mu\rangle=\omega_{0}$.

Now we claim that
\begin{eqnarray}\label{456}
&& \uppercase\expandafter{\romannumeral4}+\uppercase\expandafter{\romannumeral5}+\uppercase\expandafter{\romannumeral6}
 \\&=&\langle DF(\nu),\nu'\rangle\langle Y,\mu\rangle-\langle DF(\nu),\mu'\rangle f-\langle DF(\nu),\mu\rangle\langle Y,\nu'\rangle\nonumber
 \\&=&-\langle \nu_F,\mu\rangle\left(h(Y^{\p\S},Y^{\p\S})-\<\nabla f, Y^{\p\S}\>-\frac{\langle\nu,\bar{N}\rangle}{\langle\mu,\bar{N}\rangle}fh(Y^{\p\S},\mu)\right)-\frac{f}{\langle\mu,\bar{N}\rangle}h^{\partial B}(Y,(\nu_F)^{\p\S}),\nonumber
\end{eqnarray}
%where $(\nu_F)^{\p\S}=DF-\langle DF,\mu\rangle\mu$.

In fact, from \eqref{nu'}, \eqref{Y-mu} and \eqref{Y-1} we have
\begin{alignat}{2}\label{muFadd-99}
\uppercase\expandafter{\romannumeral4}&=\langle DF(\nu),\nu'\rangle\langle Y,\mu\rangle=-\frac{\langle\nu,\bar{N}\rangle}{\langle\mu,\bar{N}\rangle}f\langle DF(\nu),S^\S(Y^{\S})-\nabla f\rangle\\
&=-\frac{\langle\nu,\bar{N}\rangle}{\langle\mu,\bar{N}\rangle}f\left(\langle DF(\nu), \sum_\a h(Y^{\S},\tau_\a)\tau_\a-\nabla^{\p\S}f\rangle+\langle DF(\nu),\mu\rangle(h(Y^{\S},\mu)-\nabla_{\mu}f)\right),\nonumber
\end{alignat}
and
\begin{alignat}{2}\label{muFadd-0005}
\uppercase\expandafter{\romannumeral6}&=-\langle DF(\nu),\mu\rangle \langle Y,\nu'\rangle=-\langle DF(\nu),\mu\rangle \langle Y^{\S},-\nabla f+S^\S(Y^{\S})\rangle\\
%&=-\langle D F,\mu\rangle\left(-\<\nabla f, Y^{\p\S}\>+\frac{\langle\nu,\bar{N}\rangle}{\langle\mu,\bar{N}\rangle}f\nabla_{\mu}f+h(Y^{\p\S},Y^{\p\S})-2\frac{\langle\nu,\bar{N}\rangle}{\langle\mu,\bar{N}\rangle}fh(Y^{\p\S},\mu)+\frac{\langle\nu,\bar{N}\rangle^{2}}{\langle\mu,\bar{N}\rangle^{2}}f^{2}h(\mu,\mu)\right)\nonumber\\
&=-\langle DF(\nu),\mu\rangle\left(-\<\nabla f, Y^{\p\S}\>+\frac{\langle\nu,\bar{N}\rangle}{\langle\mu,\bar{N}\rangle}f\nabla_{\mu}f+h(Y^{\p\S},Y^{\p\S})-\frac{\langle\nu,\bar{N}\rangle}{\langle\mu,\bar{N}\rangle}fh(Y^{\p\S},\mu)-\frac{\langle\nu,\bar{N}\rangle}{\langle\mu,\bar{N}\rangle}fh(Y^{\S},\mu)\right).\nonumber
\end{alignat}
%where $\{\tau_\a\}_{\a=1}^{n-1}$ is an orthonormal basis of $T(\partial \Sigma)$.
Since $\mu=\frac{1}{\langle\mu,\bar{N}\rangle}\bar{N}-\frac{\langle\nu,\bar{N}\rangle}{\langle\mu,\bar{N}\rangle}\nu$, we see
\begin{eqnarray}\label{xeq5}
S_{1}(Y^{\p\S})=\frac{1}{\langle\mu,\bar{N}\rangle}\sum_\a h^{\partial B}(Y^{\p\S},\tau_\a)\tau_\a-\frac{\langle\nu,\bar{N}\rangle}{\langle\mu,\bar{N}\rangle}\sum_\a h(Y^{\p\S},\tau_\a)\tau_\a.
\end{eqnarray}
Using \eqref{mu'}, \eqref{xeq5} and \eqref{Y-2}, we deduce
\begin{alignat}{2}\label{muFadd-000}
\uppercase\expandafter{\romannumeral5}&=-\langle DF(\nu),\mu'\rangle f\\
%&=-\left\langle DF,S_{1}(Y^{\p\S})+\frac{\langle\nu,\bar{N}\rangle}{\langle\mu,\bar{N}\rangle}\nabla^{\p\S} f+\frac{\langle\nu,\bar{N}\rangle^{2}}{\langle\mu,\bar{N}\rangle^{2}}fh(\mu,\tau_\a)\tau_\a+\frac{f}{\langle\mu,\bar{N}\rangle^{2}}h^{\partial B}(\bar{\nu},\tau_\a)\tau_\a\right\rangle f\nonumber
&=-\left\langle DF(\nu),\frac{1}{\langle\mu,\bar{N}\rangle}\sum_\a h^{\partial B}(Y^{\p\S},\tau_\a)\tau_\a-\frac{\langle\nu,\bar{N}\rangle}{\langle\mu,\bar{N}\rangle}\sum_\a h(Y^{\p\S},\tau_\a)\tau_\a+\frac{\langle\nu,\bar{N}\rangle}{\langle\mu,\bar{N}\rangle}\nabla^{\p\S} f\right\rangle f\nonumber\\
&\qquad-\left\langle DF(\nu),\frac{\langle\nu,\bar{N}\rangle^{2}}{\langle\mu,\bar{N}\rangle^{2}}f \sum_\a h(\mu,\tau_\a)\tau_\a+\frac{f}{\langle\mu,\bar{N}\rangle^{2}}\sum_\a h^{\partial B}(\bar{\nu},\tau_\a)\tau_\a\right\rangle f\nonumber\\
%&=-\left\langle DF,\frac{\langle\nu,\bar{N}\rangle}{\langle\mu,\bar{N}\rangle}\nabla^{\p\S} f-\frac{\langle\nu,\bar{N}\rangle}{\langle\mu,\bar{N}\rangle}\left(h(Y^{\p\S},\tau_\a)\tau_\a-\frac{\langle\nu,\bar{N}\rangle}{\langle\mu,\bar{N}\rangle}fh(\mu,\tau_\a)\tau_\a\right)\right\rangle f\nonumber\\hysuj,
%&\qquad-\left\langle DF,\frac{1}{\langle\mu,\bar{N}\rangle}\left(h^{\partial B}(Y^{\p\S},\tau_\a)\tau_\a+\frac{f}{\langle\mu,\bar{N}\rangle}h^{\partial B}(\bar{\nu},\tau_\a)\tau_\a\right)\right\rangle f\nonumber\\
&=-\left\langle DF(\nu),\frac{\langle\nu,\bar{N}\rangle}{\langle\mu,\bar{N}\rangle}\nabla^{\p\S} f-\frac{\langle\nu,\bar{N}\rangle}{\langle\mu,\bar{N}\rangle}\sum_\a h(Y^{\S},\tau_\a)\tau_\a+\frac{1}{\langle\mu,\bar{N}\rangle}\sum_\a h^{\partial B}(Y,\tau_\a)\tau_\a\right\rangle f.\nonumber
\end{alignat}
Now the above claim \eqref{456} follows by combining \eqref{muFadd-99}, \eqref{muFadd-0005} and \eqref{muFadd-000}.

Putting $\uppercase\expandafter{\romannumeral1}$-$\uppercase\expandafter{\romannumeral6}$ into \eqref{muFadd}, we get
\begin{alignat}{2}\label{muFadd-Yf}
&\langle Y,\mu_{F}'+\omega_{0}\bar{\nu}'\rangle\\
&=f\left(\langle A_{F}\nabla f,\mu\rangle-h_{F}(Y^{\S},\mu)-\frac{1}{\langle\mu,\bar{N}\rangle}h^{\partial B}(Y,(\nu_F)^{\p\S})\right)+\langle\mu_{F},\bar{N}\rangle h^{\partial B}(Y^{\p\S},Y).\nonumber
%&=f\left(\langle A_{F}\nabla f,\mu\rangle-h_{F}(\mu,\mu)\langle Y,\mu\rangle-h_{F}(Y^{\p\S},\mu)-\frac{1}{\langle\mu,\bar{N}\rangle}h^{\partial B}(Y,\tau_\a)\langle DF,\tau_\a\rangle\right)+\langle\mu_{F},\bar{N}\rangle h^{\partial B}(Y^{\p\S},Y)\nonumber
\end{alignat}
Combining \eqref{muFadd-Yf} and \eqref{Y'}, using \eqref{Y-2} and \eqref{xeq6}, we have
\begin{alignat}{2}\label{muFadd-1oi}
&\langle Y,\mu_{F}'+\omega_{0}\bar{\nu}'\rangle+\langle Y',\mu_{F}+\omega_{0}\bar{\nu}\rangle\\
%&=f\left(\langle A_{F}\nabla f,\mu\rangle-h_{F}(Y^{\S},\mu)-\frac{1}{\langle\mu,\bar{N}\rangle}h^{\partial B}(Y,(\nu_F)^{\p\S})\right)+\langle\mu_{F},\bar{N}\rangle h^{\partial B}(Y^{\p\S}-Y,Y)\nonumber\\.
&=f\left(\langle A_{F}\nabla f,\mu\rangle-h_{F}(Y^{\S},\mu)-\frac{1}{\langle\mu,\bar{N}\rangle}h^{\partial B}(Y,(\nu_F)^{\p\S})-\frac{\langle\mu_{F},\bar{N}\rangle}{\langle\mu,\bar{N}\rangle}h^{\partial B}(Y,\bar{\nu})\right)\nonumber\\
&=f\left(\langle A_{F}\nabla f,\mu\rangle-q_{F}f\right).\nonumber
\end{alignat}
The proof is completed by \eqref{variation-formula} and \eqref{muFadd-1oi}. \qed

\

{\bf Acknowledgements.} The authors would like to thank Professor Haizhong Li and Professor Guofang Wang for their constant support and their interest on this topic. The first author would also like to thank Dr. Han Hong for interesting discussion.

\

\

\end{document}